\documentclass[11pt]{amsart} 
\usepackage{amsthm,amsbsy,amsfonts,amssymb,amscd}
\usepackage{bbm}
\usepackage[enableskew]{youngtab}
\usepackage[mathscr]{euscript} 

\usepackage{color}

\usepackage{enumerate} 
\usepackage{tikz} 
\usetikzlibrary{calc,intersections,through, backgrounds,arrows,decorations.pathmorphing,backgrounds,positioning,fit,petri} 
\usetikzlibrary{calc}

\usepackage{tikz-cd} 
\usepackage[all]{xy}
\xyoption{knot}
\xyoption{poly}

\usepackage{hyperref}

\DeclareMathAlphabet{\mathpzc}{OT1}{pzc}{m}{it}

\usepackage[margin=1.1in]{geometry}

\newtheoremstyle{notes} {} {} {} {} {\bfseries} {.} {.5em} {}

\theoremstyle{plain}

\newtheorem{prop}[subsubsection]{Proposition}

\newtheorem{lemma}[subsubsection]{Lemma}
\newtheorem{cor}[subsubsection]{Corollary}
\newtheorem{thm}[subsubsection]{Theorem}

\newtheorem{thmA}{Theorem}

\theoremstyle{remark}

\newtheorem{qu}[subsubsection]{Question} 
\newtheorem{rem}[subsubsection]{Remark} 
\newtheorem{ddef}[subsubsection]{Definition} 
\newtheorem{ex}[subsubsection]{Example} 

\newcommand{\mB}{\mathbb{B}}

\newcommand{\charr}{\mathrm{char}}

\newcommand{\ev}{\mathrm{ev}}
\newcommand{\co}{\mathrm{co}}
\newcommand{\tto}{\twoheadrightarrow}
\newcommand{\mZ}{\mathbb{Z}}
\newcommand{\mO}{\mathbb{O}}
\newcommand{\mF}{\mathbb{F}}
\newcommand{\mN}{\mathbb{N}}

\newcommand{\Funct}{\mathrm{Funct}}

\newcommand{\cO}{\mathcal{O}}
\newcommand{\cA}{\mathcal{A}}

\newcommand{\scA}{\mathscr{A}}
\newcommand{\scR}{\mathscr{R}}
\newcommand{\scB}{\mathscr{B}}

\newcommand{\op}{\mathrm{op}}
\newcommand{\Autt}{\underline{\mathrm{Aut}}^\otimes}

\newcommand{\oRepG}{\overline{\mathrm{Rep}G}}

\newcommand{\Fr}{\mathtt{Fr}}
\newcommand{\innn}{\mathrm{in}}

\newcommand{\End}{{\rm End}}

\newcommand{\mk}{\Bbbk}

\newcommand{\Hom}{{\rm Hom}}

\newcommand{\Rep}{{\rm Rep}}
\newcommand{\Nat}{{\rm Nat}}
\newcommand{\Ext}{{\rm Ext}}
\newcommand{\unit}{{\mathbbm{1}}}
\newcommand{\bC}{{\mathbf{C}}}
\newcommand{\bD}{{\mathbf{D}}}
\newcommand{\cN}{{\mathcal{N}}}
\newcommand{\cC}{{\mathcal{C}}}

\newcommand{\Ob}{{\mathrm{Ob}}}

\newcommand{\Id}{{\mathrm{Id}}}

\newcommand{\Res}{\mathrm{Res}}

\newcommand{\sgn}{\mathrm{sgn}}

\newcommand{\oa}{\bar{0}}
\newcommand{\ob}{\bar{1}}

\newcommand{\Ab}{\mathbf{A\hspace{-0.4mm}b}}
\newcommand{\vecc}{\mathbf{vec}}
\newcommand{\Vecc}{\mathbf{Vec}}

\newcommand{\Grp}{\mathbf{G\hspace{-0.13mm}r\hspace{-0.13mm}p}}
\newcommand{\Set}{\mathbf{S\hspace{-0.2mm}e\hspace{-0.2mm}t}}

\newcommand{\Spec}{\mathrm{Spec}}

\newcommand{\Aut}{\mathrm{Aut}}

\newcommand{\Mod}{\mathbf{Mod}}

\newcommand{\Dim}{\mathrm{Dim}}
\newcommand{\Alg}{\mathrm{Alg}}
\newcommand{\alg}{\mathrm{alg}}
\newcommand{\bT}{{\bf T}}
\newcommand{\bV}{{\bf V}}

\newcommand{\Sym}{\mathrm{Sym}}
\newcommand{\id}{\mathrm{id}}
\newcommand{\bA}{{\mathbf{A}}}
\newcommand{\bB}{{\mathbf{B}}}
\newcommand{\gr}{{\mathrm{gr}}}
\newcommand{\GL}{{\mathsf{GL}}}
\newcommand{\SG}{{\mathsf{S}}}
\newcommand{\CG}{{\mathsf{C}}}

\newcommand{\Ind}{{\mathrm{Ind}}}

\newcommand{\Triv}{\mathrm{Triv}}

\newcommand{\svec}{\mathbf{svec}}

\newcommand{\ver}{\mathbf{ver}}

\title[Tannakian categories]{Tannakian categories in positive characteristic} 
\author{Kevin Coulembier}



\begin{document} 
\date{}

\maketitle 



\begin{abstract}
We determine internal characterisations for when a tensor category is (super) tannakian, over fields of positive characteristic. This generalises the corresponding characterisations in characteristic zero by P. Deligne. We also explore notions of Frobenius twists in tensor categories in positive characteristic.
\end{abstract}

\section*{Introduction}
For a field $\mk$, a tensor category over $\mk$ is a $\mk$-linear abelian rigid symmetric monoidal category where the endomorphism algebra of the tensor unit is $\mk$ (contrary to some authors we do not require objects to have finite length). The standard example is the category $\Rep G$ of finite dimensional algebraic representations of an affine group scheme $G$ over $\mk$. An affine group scheme over an algebraically closed field is determined (up to inner automorphism) by its representation category, so tensor categories can be seen as a generalisation of affine group schemes. 
This motivates looking for internal characterisations which determine when a tensor category is equivalent to such a representation category.
By combining results in \cite{Del90, Del02, Del11}, we have the following theorem in characteristic zero.
\begin{thmA}[Deligne]
Let $\mk$ be an algebraically closed field of characteristic zero and $\bT$ a tensor category over $\mk$. The following are equivalent.
\begin{enumerate}[(i)]
\item As a tensor category, $\bT$ is equivalent to $\Rep G$ for some affine group scheme $G/\mk$.
\item For each $X\in\bT$ there exists $n\in\mN$ such that $\Lambda^nX=0$.
\end{enumerate}
\end{thmA}
It is known that the same statement does not hold true for fields of positive characteristic, see~\cite{BE, GK, GM}.
Fix a field $\mk$ of characteristic $p$ and a tensor category $\bT$ over $\mk$. Since the group algebra $\mk\SG_n$ is not semisimple when $n\ge p$, we need to distinguish between the symmetric power $\Sym^nX={\mathrm{H}}_0(\SG_n,\otimes^nX)$ of $X\in\bT$, which is a quotient of $\otimes^nX$, and the divided power $\Gamma^nX={\mathrm{H}}^0(\SG_n,\otimes^nX)$, which is a subobject of $\otimes^nX$. For $j\in\mN$, we define the object $\Fr_+^{(j)}X$ in $\bT$ as the image of the composite morphism
$$\Gamma^{p^j}X\;\hookrightarrow\;\otimes^{p^j}X\;\tto\; \Sym^{p^j}X.$$
The choice of notation $\Fr_+^{(j)}$ is motivated by the fact that for a vector space $V$, the space $\Fr_+^{(j)}V$ is canonically identified with the $j$-th Frobenius twist $V^{(j)}$ of $V$. Using this construction, we can now formulate our first main result, proved in Theorems~\ref{Thm3} and \ref{NewThmNeut}. For a definition of the exterior power $\Lambda^n$ we refer to Section~\ref{defSym}.
\begin{thmA}
Let $\mk$ be an algebraically closed field of positive characteristic and $\bT$ a tensor category over $\mk$. The following are equivalent.
\begin{enumerate}[(i)]
\item As a tensor category, $\bT$ is equivalent to $\Rep G$ for some affine group scheme $G/\mk$.
\item For each $X\in \bT$
\begin{enumerate}[(a)]
\item there exists $n\in \mN$ such that $\Lambda^nX=0$;
\item we have $\Lambda^nX=0$ when $\Lambda^n\Fr_+^{(j)}(X)=0$, for $j,n\in\mN$.
\end{enumerate}
\end{enumerate}
\end{thmA}
In Theorem~\ref{ThmSuper}, we prove a similar characterisation of the representation categories of affine supergroup schemes among all tensor categories.
In characteristic zero, such an internal characterisation follows from the main result of \cite{Del02}, which states that any tensor category `of sub-exponential growth' is equivalent to such a representation category.

The assignment $X\mapsto \Fr_+^{(1)}X$ actually yields an additive functor $\Fr_+:\bT\to\bT $.
The following is proved in Theorems~\ref{Thm1} and~\ref{Thm1p} and Proposition~\ref{CoolThm}.

\begin{thmA}
Let $\mk$ be a field of characteristic $p>0$ and $\bT$ a tensor category over $\mk$. The following are equivalent.
\begin{enumerate}[(i)]
\item The functor $\Fr_+:\bT\to\bT $ is exact.
\item For each filtered object $X\in\bT$, the canonical epimorphism $\Sym^\bullet (\gr X)\tto \gr (\Sym^\bullet X)$ is an isomorphism.
\item For each monomorphism $\unit\hookrightarrow X$, the induced morphism $\unit\to\Sym^pX$ is non-zero.
\item There exists an abelian $\mk$-linear symmetric monoidal category $\bC$ and an exact $\mk$-linear symmetric monoidal functor $F:\bT\to\bC$ which splits every short exact sequence in $\bT$.
\end{enumerate}
\end{thmA}
Property C(ii) is relevant in the study of the $p$-adic categorical dimensions $\Dim_{\pm}:\Ob\bT\to \mZ_p$ as defined in \cite{EHO}. If it (or hence any other property in Theorem~C) is satisfied, the $p$-adic dimensions are additive along short exact sequences.

As explained above, in this paper we study when a tensor category is equivalent to the representation category of an affine (super)group scheme. In~\cite[Conjecture~1.3]{Ostrik}, Ostrik proposed a different conjectural extension of the results in \cite{Del90, Del02}. The conjecture states that tensor categories over algebraically closed fields of characteristic $p$ which are of sub-exponential growth are equivalent to representation categories of affine group schemes in the `universal Verlinde category' Ver${}_p$.
In \cite{Ostrik} this conjecture is proved for symmetric fusion categories.
The proof relied in an essential way on a generalisation of the classical Frobenius twist to fusion categories. We prove that our functor $\Fr_+$ is a direct summand of a functor $\Fr$ which, when applied to fusion categories, recovers the functor in \cite{Ostrik}. We hope that our generalisation of Ostrik's Frobenius twist to arbitrary tensor categories might be useful in the exploration of \cite[Conjecture~1.3]{Ostrik}.

The rest of the paper is organised as follows. In Section~\ref{SecPrel} we review some properties of tensor categories. In Section~\ref{SecCat} we study (modular) representation theory of finite groups in abelian categories. This will be used later on to deal with the representations of the symmetric group, and its subgroups, which originate from the symmetric braiding on tensor categories. In Section~\ref{LSSTC} we define and study `locally semisimple' tensor categories, which are the ones in which the equivalent conditions in Theorem~C are satisfied. We also derive abstract criteria for existence of tensor functors to semisimple tensor categories.  In Section~\ref{SecFrob}, we study the Frobenius twists. In Section~\ref{SecTanObj} we give internal characterisations for objects which are `locally free', that is objects which become isomorphic to (super) vector spaces after internal extension of scalars. As a consequence of those results we obtain internal characterisations of (super) tannakian categories (as defined in \ref{DefTann}) in Section~\ref{SecTanCat}. We also show that each tensor category has a unique maximal (super) tannakian subcategory. Finally we prove the result, announced in \cite{Del02}, that over algebraically closed fields, super tannakian categories are always representation categories of affine supergroup schemes, by adapting an unpublished argument from Deligne in \cite{Del11} using some results of Section~\ref{LSSTC}.

\section{Preliminaries and notation}\label{SecPrel}

Unless further specified, $\mk$ denotes an arbitrary field. We set $\mN=\{0,1,2,\ldots\}$.

\subsection{Symmetric and cyclic groups}
For a finite group $G$ we denote by $\Rep_{\mk}G$ the category of finite dimensional $\mk G$-modules.
\subsubsection{The symmetric and cyclic group}\label{SecSG} We denote the symmetric group acting on $\{1,2,\ldots, n\}$ by $\SG_n$. This defines a canonical embedding $\SG_m<\SG_n$ for $m<n$. 

\label{SecCG}We denote by $\CG_n$ the cyclic group of order $n$ and we fix the embedding $\CG_n< \SG_n$ which maps the generator of $\CG_n$ to the cycle $(1,2,\ldots,n)\in \SG_n$. 
Assume that $\charr(\mk)=p>0$. We denote by $M_i$ the indecomposable $\mk\CG_p$-module of dimension~$i$, for $1\le i\le p$.  In particular $M_1\simeq\mk$ and $M_p\simeq\mk\CG_p$.
 Every object in $\Rep\CG_p$ is a direct sum of these modules.

\subsubsection{Wreath products}\label{pgroups} Fix a prime number $p$ and $a\in\mZ_{>0}$. We define subgroups of $\SG_{p^a}$
$$\iota_a:\SG_{p^{a-1}} \hookrightarrow \SG_{p^a}\quad\mbox{and}\quad \iota_a':\SG_p\hookrightarrow \SG_{p^a},$$
as follows. For every $i\in\{1,\ldots, p^a\}$, there are unique $1\le m\le p^{a-1}$ and $1\le n\le p$ such that $i=m+(n-1)p^{a-1}$. Then $\iota_a'(\SG_p)$ acts on $n\in\{1,\ldots, p\}$ and $\iota_a(\SG_{p^{a-1}})$ acts on $m$.


For $j\in\mZ_{>0}$, we define the subgroups $P_j<\SG_{p^j}$ and $Q_j<\SG_{p^j}$ iteratively. We set $ P_1=\CG_p$, $Q_1=\SG_p$ and let $P_{j+1}$ be generated by $\iota'_{j+1}(\CG_p)$ and $P_j$ interpreted as $P_j<\SG_{p^j}<\SG_{p^{j+1}}$, with the latter inclusion as in \ref{SecSG}. Similarly, $Q_{j+1}$ is generated by $Q_j$ and $\iota'_{j+1}(\SG_p)$. In other words
$$P_{j+1}:=P_j\wr \CG_p\simeq P_j^{\times p} \rtimes \CG_p\quad\mbox{and}\quad Q_{j+1}:=Q_j\wr \SG_p\simeq Q_j^{\times p} \rtimes \SG_p,$$

\begin{lemma}\label{LemNormSyl}
The group $P_j$ is a Sylow $p$-subgroup of $\SG_{p^j}$, and $Q_j$ contains its normaliser $N_{\SG_{p^j}}(P_j)$.
\end{lemma}
\begin{proof}
That $P_j$ is a Sylow subgroup is well-known, see \cite{Ber}, and follows immediately from Legendre's theorem on the prime factorisation of factorials. We set $N_j:=N_{\SG_{p^j}}(P_j)$. It is also proved {\it loc. cit.} that
$|N_j:P_j|=(p-1)^{j}.$
In the proof of $N_j< Q_j$  we will freely use the fact that the images of $\iota_j$ and $\iota_j'$ are commuting subgroups which generate a copy of $\SG_{p^{j-1}}\times \SG_p$ in $\SG_{p^{j}}$, and that $\iota_j(Q_{j-1})<Q_j>\iota'_j(\SG_p)$.

We define iteratively subgroups $M_j<Q_j$. We set $M_1=N_1$, so $|N_1|=p(p-1)$, and $M_j=\iota_j(M_{j-1})\times \iota'_j(N_{1})$. By construction, $M_j$ normalises $P_j$ (meaning $M_j<N_j$) and satisfies $|M_j|=(p(p-1))^j$.
By Lagrange's theorem, the group generated by $M_j$ and $P_j$ has order divisible by $(p-1)^j|P_j|=|N_j|$. Hence the group coincides with $N_j$ which concludes the proof.
\end{proof}

\subsection{Monoidal categories}

\subsubsection{Categories}When clear in which category we are working, we will denote the morphism sets simply by $\Hom$, $\End$ or $\Aut$. 
For $\mk$-linear categories $\bA$ and $\bB$, we denote by $\bA\otimes_{\mk}\bB$ the $\mk$-linear category with objects $(X,Y)$ for $X\in\bA$ and $Y\in\bB$ and the space of morphisms from $(X,Y)$ to $(Z,W)$ given by $\bA(X,Z)\otimes_{\mk}\bB(Y,W)$. Then we denote by $\bA\boxtimes_{\mk}\bB$, or $\bA\boxtimes\bB$, the Karoubi envelope of $\bA\otimes_{\mk}\bB$. The object $(X,Y)$ as considered in $\bA\boxtimes\bB$ will be written as $X\boxtimes Y$.

An abelian $\mk$-linear category in which the endomorphism algebra of each simple object is $\mk$ is called {\bf schurian}. A semisimple schurian category is thus equivalent to a direct sum of copies of the category $\vecc_{\mk}$ of finite dimensional vector spaces.
If $\bA$ and $\bB$ are $\mk$-linear abelian with $\bB$ semisimple schurian, then $\bA\boxtimes \bB$ is again abelian.

An object $X$ in an abelian category with subobjects
$$0=X_{0}\subset X_1\subset X_2\subset\cdots\subset X_d=X$$
will be called an object with filtration of length $d$. The associated graded object is
$$\gr X =\oplus_{i=1}^d \gr_i X,\quad\mbox{with}\quad \gr_i=X_i/X_{i-1}.$$
Following, \cite[\S I.8.2]{AGV}, for a locally small category $\bC$, we denote by $\Ind \bC$ the full subcategory of the category of functors $\bC^{\op}\to\Set$ consisting of ind-objects.

\subsubsection{}\label{SMC} We will work with symmetric monoidal categories $(\bC,\otimes,\unit,\gamma)$ where
\begin{enumerate}[(i)]
\item $\bC$ is $\mk$-linear abelian (with $\unit\not=0$);
\item $-\otimes -$ is right exact and $\mk$-linear in both variables.
\end{enumerate}
Here $\gamma$ refers to the binatural family of braiding morphisms $\gamma_{XY}:X\otimes Y\stackrel{\sim}{\to} Y\otimes X$ which satisfy the constraints of \cite[\S 1]{DM}. For $X\in\bC$ and $n\in\mZ_{\ge 1}$, we write
$$\otimes^nX\;=\;\stackrel{n}{\overbrace{X\otimes X\otimes\cdots\otimes X}}\qquad\mbox{and}\qquad \otimes^0X\;=\;\unit,$$
and use similar notation for morphisms. An object $X$ in $\bC$ is {\bf flat} if $X\otimes-:\bC\to\bC$ is exact.

\subsubsection{}\label{dual} Let $\bC$ be as in \ref{SMC}. For an object $X\in \bC$, a dual $X^\vee$ is an object equipped with morphisms $\co_X:\unit\to X\otimes X^\vee$ and $\ev_X:X^\vee\otimes X\to\unit$ satisfying the two snake relations in \cite[(0.1.4)]{Del02}. Following \cite[\S 1.4]{Del02}, we then have bi-adjoint functors $(-\otimes X,-\otimes X^\vee)$. In particular, dualisable objects are flat. For dualisable $X,Y\in\bC$, there is an isomorphism
$$\Hom(X,Y)\stackrel{\sim}{\to}\Hom(Y^\vee,X^\vee),\quad f\mapsto f^t=(\ev_Y\otimes \Id_{X^\vee})\circ (\Id_{Y^\vee}\otimes f\otimes \Id_{X^\vee})\circ (\Id_{Y^\vee}\otimes \co_X).$$
A direct summand of a dualisable object is also dualisable, see~\cite[\S 1.15]{Del02}.

\subsubsection{}\label{defTensor}Following \cite{Del02, Del90}, $\bC$ as in \ref{SMC} is a {\bf tensor category over $\mk$} if additionally it is essentially small and
\begin{enumerate}[(i)] \setcounter{enumi}{2}
\item the canonical morphism $\mk\to\End(\unit)$ is an isomorphism;
\item every object in $\bC$ is dualisable.
\end{enumerate}
Now let $\bT$ be a tensor category. By \ref{dual}, the functor $-\otimes -$ is bi-exact  and by \cite[Proposition~1.17]{DM}, the unit object $\unit$ is simple.
If every object has finite length, then every morphism space is finite dimensional, see~\cite[Proposition~1.1]{Del02}.
If $\mk$ is algebraically closed and every object in $\bT$ has finite length, $\bT$ is therefore schurian.

\subsubsection{} A right exact $\mk$-linear functor $F:\bC\to \bC'$ between two categories $\bC$ and $\bC'$ as in \ref{SMC} is a {\bf  tensor functor} if it is equipped with natural isomorphisms $c_{XY}^F:F(X)\otimes F(Y)\stackrel{\sim}{\to}F(X\otimes Y)$ and $F(\unit)\stackrel{\sim}{\to}\unit$ satisfying the compatibility conditions of \cite[\S 2.7]{Del90}, see also \cite[Definition~1.8]{DM}. In particular tensor functors are {\em symmetric} monoidal functors. 
We will usually just write $F$ for the tensor functor, rather than $(F,c^F)$. The following lemma, see \cite[Corollaire~2.10]{Del90}, is a straightforward consequence of the definitions and the fact that $\unit$ is simple in a tensor category. 

\begin{lemma}\label{LemFaith}
Consider a tensor functor $F:\bC\to\bC'$. If $X \in\bC$ has a dual $X^\vee$ then $F(X^\vee)$ is a dual of $F(X)$.
Any tensor functor out of a tensor category is exact and faithful.
\end{lemma}

\subsubsection{}\label{Defeta} For categories $\bC,\bC'$ as in \ref{SMC} and tensor functors $F,G:\bC\to\bC'$, a natural transformation $F\Rightarrow G$ is one of tensor functors if it exchanges the monoidal structures as in \cite[\S 2.7]{Del90}. As pointed out {\it loc. cit.}, for such $\eta:F\Rightarrow G$ and dualisable $X\in \bC$,
the morphisms $\eta_X$ and $(\eta_{X^\vee})^t$ are mutually inverse.
In particular, a natural transformation of tensor functors out of a tensor category is {\em automatically an isomorphism,} see also~\cite[Proposition~1.13]{DM}.

\subsubsection{} \label{AlgC}
For $\bC$ as in \ref{SMC}.
We denote the category of commutative algebras in $\bC$ by $\alg{\bC}$.
Such an algebra is a triple $(\scA,m,\eta)$, with $\scA$ an object in $\bC$ and morphisms $m:\scA\otimes\scA\to\scA$ and $\eta:\unit\to\scA$ satisfying the traditional commutative (with respect to $\gamma$) algebra relations.

For non-zero $\scA\in \alg{\bC}$, we denote the category of $\scA$-modules in $\bC$ by $\bC_{\scA}$. Then $(\bC_{\scA},\otimes_{\scA},\scA)$ is a monoidal category as in \ref{SMC} with right exact tensor product $-\otimes_{\scA}-$, as the coequaliser of $-\otimes\scA\otimes-\rightrightarrows -\otimes-$, introduced in  \cite[\S 7.5]{Del90}. For an algebra morphism $\scA\to\scB$ we have the corresponding tensor functor $\scB\otimes_\scA-:\bC_\scA\to\bC_{\scB}$.


\subsubsection{} For the remainder of this subsection, fix a tensor category $\bT$. By \cite[\S 7.5]{Del90} the category $\Ind \bT$ is again naturally a symmetric monoidal category satisfying (i)-(iii) above. Furthermore, the functor $-\otimes-$ is bi-exact (and cocontinuous), even though only objects in the subcategory $\bT$ are dualisable. The latter claim can easily be proved by showing that dualisable objects must be compact, or by an immediate generalisation of the proof in \cite[\S 2.2]{Del02} of the special case where all objects in $\bT$ have finite length. 
We abbreviate the notation of \ref{AlgC} as $\Alg{\bT}=\alg\Ind\bT$ and $\Mod_{\scA}=(\Ind\bT)_{\scA}$, or $\Mod_{\scA}^{\bT}$ when there is risk of ambiguity.  We have a tensor functor
\begin{equation}\label{extsca}F_{\scA}=\scA\otimes-:\;\bT\to\Mod_{\scA}.\end{equation}

For the tensor category $\vecc=\vecc_{\mk}$, we have that $\Vecc=\Ind\vecc$ is the category of all vector spaces and $\Alg\vecc$ is the category $\mathbf{Alg}_{\mk}$ of commutative $\mk$-algebras.


A non-zero algebra $\scA$ is {\bf simple} if its only ideal subobjects are $0$ and $\scA$ itself. This is equivalent to $\scA$ being simple as an object in $\Mod_{\scA}$.

\begin{lemma}\label{LemMaxId}
Every non-zero algebra $\scA$ in $\Alg\bT$ has a simple quotient.
\end{lemma}
\begin{proof}
Since $\bT$ is essentially small, $\Ind\bT$ has a generator and the class of subobjects of any object in $\Ind\bT$ forms a set. By Zorn's lemma, $\scA$ contains a maximal ideal subobject $J$. The algebra $\scB=\scA/J$ is simple.
\end{proof}

\subsubsection{Tensor subcategories}\label{TensSub} A full subcategory $\bT'$ of $\bT$ is a tensor subcategory if it is closed under the operations of taking subquotients, tensor products, duals and direct sums. In particular $\bT'$ is replete in $\bT$ and a tensor category itself. For $E$ a collection of objects or full subcategories of $\bT$, we denote by $\langle E\rangle$ the minimal tensor subcategory of $\bT$ which contains all objects in $E$. We say that $\bT$ is {\bf finitely generated} if $\bT=\langle X\rangle$, for some object $X\in\bT$.

We denote by $\Gamma^{\bT}_{\bT'}$, or simply $\Gamma_{\bT'}$, the right adjoint to the inclusion functor $\Ind\bT'\to\Ind\bT$. In other words, $\Gamma_{\bT'}$ is the left exact lax monoidal functor which sends an object in $\Ind\bT$ to its maximal subobject in $\Ind\bT'$.

If $\bT$ and $\bV$ are tensor categories where $\bV$ is semisimple schurian, then $\bT\boxtimes\bV$ is again a tensor category. We can and will identify $\bT$ and $\bV$ with tensor subcategories of $\bT\boxtimes\bV$.

The following observations are standard, see e.g. \cite[\S 2.11]{Del02}.
\begin{lemma}\label{LemLax}
Consider a tensor subcategory $\bT'\subset\bT$. 
\begin{enumerate}[(i)]
\item If $\scA$ is in $\Alg\bT$, then $\Gamma_{\bT'}\scA$ is a subalgebra, and hence an object in $\Alg\bT'$.
\item For $X\in\Ind\bT$ and $Y\in\bT'$, the natural morphism $\Gamma_{\bT'}(X)\otimes Y\to\Gamma_{\bT'}(X\otimes Y)$ is an isomorphism.
\item For $\scA\in\Alg\bT$, set $\scR:=\Gamma_{\bT'}\scA$. The lax monoidal structure of $\Gamma_{\bT'}:\Ind\bT\to\Ind\bT'$ induces one as a functor $\Gamma_{\bT'}:\Mod_{\scA}^{\bT}\to\Mod_{\scR}^{\bT'}$ via the universality of the coequaliser $-\otimes_{\scR}-$. Furthermore,
$$\Gamma_{\bT'}(M)\otimes_{\scR}\Gamma_{\bT'}(N)\;\to\; \Gamma_{\bT'}(M\otimes_{\scA}N)$$
is an isomorphism if $M\simeq \scA\otimes M_0$ and $N\simeq\scA\otimes N_0$ for $M_0,N_0\in\bT'$.
\end{enumerate}
\end{lemma}

\subsection{Symmetric and divided powers}\label{defSym}
Let $\bC$ be a symmetric monoidal category as in \ref{SMC} and fix $X\in\bC$.

\subsubsection{} We define $\Lambda^2X$ as the image of the morphism $\gamma_{XX}-1$ in $\End({\otimes^2X})$. By definition, $\Lambda^2X$ is a subobject of $\otimes^2 X$. If $2\not=\charr(\mk)$, then $\Lambda^2X$ is a direct summand. If $\bC$ admits countable coproducts, the symmetric algebra
$$\Sym^\bullet X\;=\;\bigoplus_{i\in\mN}\Sym^iX\;\in\;\alg{\bC}$$
is the maximal commutative quotient of the tensor algebra of $X$. It is therefore the quotient with respect to the ideal generated by $\Lambda^2X$. Concretely, for $n\in\mN$, $\Sym^nX$ is the cokernel of
\begin{equation}\label{eqSym} \bigoplus_{a+b=n-2} \otimes^{a}X\otimes \Lambda^2X\otimes \otimes^{b}X\to  \otimes^nX, \quad\mbox{ or }\quad\bigoplus_{i=1}^{n-1} \otimes^nX\to \otimes^nX,\end{equation}
where the used endomorphisms of $\otimes^nX$ are $\otimes^{i-1}X\otimes(\gamma_{XX}-1)\otimes \otimes^{n-i-1}X$.

Dually we define the divided power $\Gamma^nX$ by the exact sequence
$$0\; \to\;\Gamma^nX\;\to\; \otimes^nX\;\to\; \bigoplus_{i=1}^{n-1} \otimes^nX$$
using the same endomorphisms of $\otimes^nX$. Equivalently we can set $\Sym^nX={\mathrm{H}}_0(\SG_n,\otimes^nX)$ and $\Gamma^nX={\mathrm{H}}^0(\SG_n,\otimes^nX)$, with notation as in Section~\ref{DefSecCat}.

If $X$ is dualisable and $\Sym^nX$ is flat, $\Sym^nX$ is dualisable with dual $\Gamma^n(X^\vee)$. This follows from \eqref{eqSym} and the following lemma.

\begin{lemma}
If the cokernel $C$ of $u:A\to B$ is flat for dualisable $A,B\in\bC$, then $C$ is dualisable with dual given by the kernel $K$ of $v:=u^t:B^\vee\to A^\vee$.
\end{lemma}
\begin{proof}
It follows by definition of $v$ that the composition $K \otimes B \hookrightarrow  B^\vee \otimes B  \stackrel{\ev}{\to}\unit$ factors as $K \otimes B \tto K \otimes  C \stackrel{\varepsilon}{\to}\unit$, for some morphism $\varepsilon$. The flatness of $ C $ allows to conclude that the bottom vertical arrow in the diagram 
$$\xymatrix{
\unit\ar[rr]^{\co}\ar@{-->}[d]^{\delta}&&B \otimes  B^\vee \ar@{->>}[d] \\
 C \otimes K \ar@{^{(}->}[rr]&& C \otimes  B^\vee 
}$$
 is a monomorphism. 
The existence of a morphism $\delta$ to create a commutative diagram then follows as for $\varepsilon$. That $\varepsilon$ and $\delta$ satisfy the snake relations now follows from construction. For instance, the commutative diagram
$$\xymatrix{
 B^\vee \ar@{^{(}->}[rr]^-{ B^\vee \otimes\co}&& B^\vee \otimes B \otimes  B^\vee \ar@{->>}[rr]^-{\ev\otimes  B^\vee }&&  B^\vee \\
&& K \otimes B \otimes  B^\vee \ar@{^{(}->}[u]\ar@{->>}[d]\\
&& K \otimes C \otimes  B^\vee \ar@/_/[uurr]_-{\varepsilon\otimes  B^\vee }\\
 K \ar@{^{(}->}[uuu]\ar[rr]^-{ K \otimes\delta}\ar@/^/[uurr]^-{ K \otimes\co}&& K \otimes C \otimes K \ar[u]\ar[rr]^-{\varepsilon\otimes K }&&K\ar@{^{(}->}[uuu]
}$$
ensures that one of the snake relations is inherited from $(\ev,\co)$.
\end{proof}

\subsubsection{Skew symmetric powers} For $n\in\mN$, we define $\Lambda^nX$ as the image of the anti-symmetriser $\sum_{\sigma\in\SG_n}(-1)^{|\sigma|}\sigma:\otimes^nX\to \otimes^nX$. The essential properties for us are $\Lambda^n\unit=0$ for $n>1$, that $\Lambda^nZ=0$ implies $\Lambda^{n+1}Z=0$, and 
\begin{equation}
\label{PropLam}
\Lambda^n(X\oplus Y)\;\simeq\;\bigoplus_{i=0}^n \Lambda^{n-i}X\,\otimes\,\Lambda^i Y,
\end{equation}
which is satisfied whenever $\Lambda^j X$ and $\Lambda^j Y$ are flat for $1\le j\le n$, or when $X=\unit$.

Now assume $\charr(\mk)\not=2$. Then we could alternatively use ${\mathrm{H}}^0(\SG_n,\sgn\otimes\otimes^nX)$ or ${\mathrm{H}}_0(\SG_n,\sgn\otimes\otimes^nX)$, with notation of Section~\ref{SecCat}, as notion of skew symmetric power and all our results remain valid with these alternative definitions. These two powers satisfy \eqref{PropLam} without any flatness condition. Note that already for $\bC=\svec$ (see \ref{svec}) with $\charr(\mk)>2$, these (co)invariants differ from the $\Lambda^\bullet$-definition.

\subsection{Semisimplification and the universal Verlinde category}\label{Verp}

\subsubsection{Semisimplification}\label{semisimp} For a tensor category $\bT$, let $\cN$ denote the ideal of negligible morphisms, see \cite[\S 7.1]{AK}, which is the unique maximal tensor ideal. We have a canonical symmetric monoidal $\mk$-linear functor $X\mapsto \overline{X}$ from $\bT$ to the quotient $\overline{\bT}:=\bT/\cN$, which maps an object to itself and a morphism to its equivalence class. As a special case of \cite[Th\'eor\`eme~8.2.2]{AK}, we find that $\oRepG$ is abelian semisimple, for a finite group $G$. We stress that $X\mapsto \overline{X}$ is in general (for instance in the example $\Rep\CG_p\to\ver_p$ below) not exact, so not a tensor functor. 

\subsubsection{The category of super vector spaces}\label{svec} Assume $\charr(\mk)\not=2$. The monoidal category $\svec$ is defined as the category of $\mZ/2$-graded vector spaces, or equivalently as $\Rep\CG_2$. The braiding is defined via the graded isomorphisms
$$\gamma_{VW}:\;V\otimes_{\mk} W\;\to\;W\otimes_{\mk}V,\; \;v\otimes w\mapsto (-1)^{|v||w|}w\otimes v,$$
where $|v|\in\mZ/2$ denotes the parity of a homogeneous vector. We denote the one-dimensional super space concentrated in odd degree by $\bar{\unit}$.

\subsubsection{Verlinde category} \label{VerpSub}Assume that $p:=\charr(\mk)>0$. In \cite[Definition~3.1]{Ostrik}, the universal Verlinde category is defined as $\ver_p:=\overline{\Rep \CG_p}$. With notation as in \ref{SecCG}, the simple objects of $\ver_p$ correspond, up to isomorphism, to $\overline{M}_i$, for $1\le i<p$. 
If $p>2$, by \cite[Proposition~2.4]{EVO},  we have
\begin{equation}\label{GammaVer}\Gamma^{p-j+1} (\overline{M}_{j})\;=\;\Sym^{p-j+1} (\overline{M}_{j})\;=\;0,\quad\mbox{for all $1<j<p$}. \end{equation}

\subsection{Fibre functors}\label{SecFibre}

For the entire subsection we consider tensor categories $\bT,\bV$, with $\bV$ schurian and semisimple. Sometimes we will require the additional assumption, satisfied for instance by $\svec$ and $\ver_p$, that
\begin{equation}
\label{CondVer}\mbox{for every simple object $\bV\ni S\not\simeq \unit$ there exists $N\in\mN$ for which $\Sym^NS=0$.}
\end{equation}
We recall the following definition from \cite[\S 3.1]{Del02}.
\begin{ddef}\label{DefFibre}
Assume that $\bV$ satisfies \eqref{CondVer}. A {\bf fibre functor of $\bT$ over $\scR$}, for a non-zero $\scR$ in $\Alg{\bV}$, is a tensor functor
$\bT\to\Mod_{\scR}^{\bV}.$
\end{ddef}

\begin{lemma}\label{SimQuoField}
If $\bV$ satisfies \eqref{CondVer}, then every simple algebra $\scA$ in $\Alg\bV$ is a field extension of $\mk$, interpreted in $\Ind\bV\supset\Vecc$.
\end{lemma}
\begin{proof}
In order to find a contradiction, we assume that there exists a simple subobject $S\not\simeq \unit$ of $\scA$. Denote by $J$ the ideal generated by $S$, meaning the image of
$$\scA\otimes S\;\hookrightarrow\; \scA\otimes\scA \;\stackrel{ m}{\to} \; \scA.$$
Since the $n$-fold multiplication $\otimes^n\scA\to\scA$ factors through $\Sym^n\scA$ we find by \eqref{CondVer} that $J$ is nilpotent in the sense that any subobject in $\bV$ of $J$ is sent to zero when multiplied (inside $\scA$) with itself enough times.
In particular, $J$ does not contain the image of $\eta:\unit\to\scA$, so $J\not=\scA$, which means $J=0$, a contradiction. In conclusion, $\scA$ is contained in $\Vecc$ and hence a commutative simple $\mk$-algebra.
\end{proof}

\begin{rem}\label{RemTriv}
Consider tensor categories $\bT_1,\bT_2$ and non-zero algebras $\scR,\scR'\in\Alg\bT_2$ with an algebra morphism $\scR\to\scR'$. The composition of any tensor functor $F:\bT_1\to\Mod_{\scR}^{\bT_2}$ with $\scR'\otimes_{\scR}-$ yields, by definition, a tensor functor $\bT_1\to\Mod_{\scR'}^{\bT_2}$.
\end{rem}

\begin{lemma}\label{LemFinLen}
If $\bT$ admits a fibre functor as in Definition~\ref{DefFibre} then we have the following.
\begin{enumerate}[(i)]
\item Each object in $\bT$ has finite length.
\item There exists a tensor functor $\bT\to\bV'$ for $\bV'$ the semisimple tensor category over $K$, for some field extension $K/\mk$, given by
 $\bV'=\vecc_K\boxtimes_{\mk}\bV$.\end{enumerate} 
\end{lemma}
\begin{proof}
Part (i) is a direct consequence of part (ii) and the fact that fibre functors are faithful, see Lemma~\ref{LemFaith}. Now we prove part (ii). Assume we have a fibre functor $\bT\to\Mod_{\scR}^{\bV}$.
By Remark~\ref{RemTriv} and Lemmata~\ref{LemMaxId} and~\ref{SimQuoField}, we can replace $\scR$ by a field extension $K/\mk$. It follows quickly that $\Mod_{K}^{\bV}\simeq \Ind\bV'$. Since the only dualisable objects in $\Ind\bV'$ are in $\bV'$, the lemma follows. \end{proof}

\begin{rem} 

In \cite[\S 3.1]{Del02} the condition that $\bV$ be semisimple is not required, but it is assumed that  all objects in $\bT$ and $\bV$ have finite length. By Lemma~\ref{LemFinLen}(i) we thus find that our notion of fibre functor is a special case of the one {\it loc. cit.}

\end{rem}

\subsubsection{}\label{DefTann} In case we take $\bV=\vecc$ in Definition~\ref{DefFibre} we recover the classical notion of a fibre functor of \cite[\S~1.9]{Del90}. A tensor category with such a fibre functor is a {\bf tannakian category}, see \cite[\S 2.8]{Del90}. When the $\mk$-algebra $\scR$ is simply $\mk$, meaning we have a tensor functor to $\vecc$, the category is {\bf neutral tannakian}. 
A tensor category admitting a fibre functor over an algebra in $\bV=\svec$ is a {\bf super tannakian category}, see \cite[\S 0.9]{Del02}. Neutral super tannakian categories are defined similarly.

\subsubsection{}For the reader's convenience we recall some essential facts about the (neutral) tannakian formalism from \cite[\S 8]{Del90}, in our limited generality. We need a couple of definitions first.  An affine group scheme in $\bT$ is a functor
$$G:\;\Alg\bT\,\to\,\Grp,$$
which is representable by a commutative Hopf algebra $\mk[G]$ in $\Ind\bT$. For a tensor functor $F:\bT\to\bT'$ to a second tensor category $\bT'$, the group functor $\Autt(F)$ sends an algebra $\scR$ in $\Alg\bT'$ to the group of automorphisms of the tensor functor $(\scR\otimes-)\circ F:\bT\to\Mod_{\scR}^{\bT'}$, and is an affine group scheme in $\bT'$ represented by the co-end algebra of $F(-)^\vee\otimes F(-): \bT^{\op}\times \bT\to\bT'$. As an example, we have the `fundamental group' $\pi(\bT)=\Autt(\Id_{\bT})$.

\begin{lemma}[Deligne]\label{LemTannForm}
Consider a tensor functor $\omega_0:\bT\to\bV$.
\begin{enumerate}[(i)]
\item  Consider the group scheme $G:=\Autt(\omega_0)$ with the canonical homomorphism $p:\pi(\bV)\to G$. Let $\Rep(G,p)$ be the category of representations of $G$ in $\bV$ which restrict to the evaluation action of $\pi(\bV)$ under $p$. We have a tensor equivalence $\bT\stackrel{\sim}{\to}\Rep(G,p)$, which yields a commutative diagram with $\omega_0$ and the forgetful functor $\Rep(G,p)\to\bV$.
\item There exists a tensor equivalence $\bT\boxtimes\bV\stackrel{\sim}{\to}\Rep G$, which yields a commutative diagram with $\omega:\bT\boxtimes\bV\to\bV$ induced from $(\omega_0,\Id_{\bV})$ and the forgetful functor $\Rep G \to\bV$.

\end{enumerate}
\end{lemma}
\begin{proof}
Part (i) is \cite[Th\'eor\`eme~8.17]{Del90}, using \cite[(8.15.1)]{Del90}. Part (ii) is \cite[Proposition 8.22]{Del90}. \end{proof}

We end this section with a result which is implicit in Sections 3 and 4 of \cite{Del02}.
\begin{prop}[Deligne]\label{PropUnique}
If $\mk$ is algebraically closed, $\bV$ satisfies \eqref{CondVer} and $\bT$ is finitely generated, then any two fibre functors $\bT\to\bV$ are isomorphic.
\end{prop}
\begin{proof}
As proved in \cite[\S 3]{Del02}, for two fibre functors $F,G:\bT\to\bV$, the co-end 
$$\Lambda\,:=\,\int^{X\in\bT}G(X)^\vee\otimes F(X)\;\in\,\Ind\bV$$
inherits the structure of an algebra in $\bV$. Furthermore $F,G$ become isomorphic as tensor functors after composition with $\scR\otimes-$, for $\scR\in\Alg\bV$, if and only if there exists an algebra morphism $\Lambda\to\scR$. Now if $\bT$ is finitely generated, so is the algebra $\Lambda$. 

By Lemma~\ref{LemMaxId}, $\Lambda$ has a simple quotient $\scA$, which by Lemma~\ref{SimQuoField} is a field extension $K$ of $\mk$. However, since $\Lambda$ is finitely generated, so is $K$, which means $K=\mk$. In particular, we have an algebra morphism $\Lambda\to\mk=\unit$. This means that $F$ and $G$ are indeed isomorphic.
\end{proof}

 



\section{Representations in abelian categories}\label{SecCat}
We fix an abelian category $\bA$, finite groups $H<G$ and a field $\mk$. 

\subsection{Definitions}\label{DefSecCat}

We will interpret groups as categories with one object where all morphisms are isomorphisms.

\begin{ddef}
A $G$-object in $\bA$ is a functor $G\to \bA$. The abelian category of such functors is denoted by $\Rep(G,\bA)$, the morphism groups by $\Hom_G$ and the forgetful functor by $\Res^G_{\ast}:\Rep(G,\bA)\to\bA$. 
\end{ddef}

Concretely, a $G$-object is of the form $X=(X_0,\phi_X)$, with $X_0=\Res^G_\ast(X)\in\bA$ and $\phi_X: g\mapsto \phi_X^g$ a group homomorphism $G\to\Aut(X_0)$. A morphism $X\to Y$ in $\Rep(G,\bA)$ is a morphism $f:X_0\to Y_0$ in $\bA$ such that $f\circ\phi^g_X=\phi_Y^g\circ f$ for all $g\in G$. 
We obtain a group homomorphism
\begin{equation}\label{GFrg}G\to\Aut(\Res^G_\ast)\,:\; g\mapsto \phi^g,\qquad\mbox{with $(\phi^g)_X=\phi^g_X$ for all $X\in\bA$.}\end{equation}
For $X,Y$ in $\Rep (G,\bA)$, the morphism group $\Hom_G(X,Y)$ can thus be interpreted as the invariants $\Hom(X_0,Y_0)^G$, for the adjoint $G$-action.
We have a fully faithful exact functor
$$\iota_{\bA}:\bA\hookrightarrow \Rep(G,\bA),\quad Y\mapsto (Y,\phi_Y)\;\mbox{with}\; \phi_Y^g:=\id_Y\;\mbox{ for all $g\in G$}. $$
We will often omit the functor $\Res^G_\ast$ and the similarly defined $\Res^G_H$ to simplify notation.

\begin{ex}
We have $\Rep(G,\vecc_{\mk})=\Rep_{\mk} G$.
\end{ex}

\begin{ddef}
Assume that $\bA$ is $\mk$-linear.
For $(M,\rho)\in \Rep_{\mk}G$ with $d=\dim_{\mk}M$ and $X\in\Rep(G,\bA)$, we define $Y:=M\otimes X$ as an object in $\Rep(G,\bA)$ with $Y_0:=\bigoplus_{i=1}^d X_{0}^{(i)}$ for objects $X_0^{(i)}$ in $\bA$ with fixed isomorphisms $\alpha_i:X_0\stackrel{\sim}{\to} X_0^{(i)}$. Furthermore, we write endomorphisms of $Y_0$ in $\bA$ as matrices and set
$$\phi_Y^g\;=\; \left(\rho(g)_{ij} (\alpha_i\circ\phi^g_X\circ\alpha_j^{-1})\right)_{1\le i,j\le d},$$
where $\rho(g)_{ij}\in\mk$ are the matrix elements of $\rho$ with respect to some fixed basis of $M$.
\end{ddef}
Alternatively we can define $M\otimes X$ as the object in $\Rep(G,\bA)$ representing the functor
$$\left(M^\ast\otimes \Hom(X_0,-)\right)^G:\;\Rep(G,\bA)\to\Ab.$$
We then easily find
\begin{equation}
\label{adjM}
\Hom_G(M\otimes X,Y)\;\simeq\;\Hom_G(X, M^\ast\otimes Y).
\end{equation}

If $\bA$ is $\mk$-linear, there is a fully faithful $\mk$-linear functor
$$\bA\,\boxtimes\, \Rep_{\mk}G\;\,{\to}\,\; \Rep(G,\bA),\quad X\boxtimes M\mapsto M\otimes \iota_{\bA}(X).$$
If $\bA$ is also semisimple and schurian then this functor is clearly an equivalence.

\begin{ddef}
\begin{enumerate}[(i)]
\item
The right and left adjoint functors of $\iota_{\bA}$ are denoted by 
$$ {\mathrm{H}}^0(G,-):\Rep(G,\bA)\to \bA \quad\mbox{and}\quad  {\mathrm{H}}_0(G,-):\Rep(G,\bA)\to \bA.$$ Concretely, $ {\mathrm{H}}^0(G,-)$ maps $X$ to the maximal subobject of $X_0$ on which each $\phi_X^g$ acts as the identity, for all $g\in G$, and $ {\mathrm{H}}_0(G,-)$ is defined dually. In symbols this gives
$$ {\mathrm{H}}^0(G,X)\;=\;\bigcap_{g\in G}\ker(\Id_{X_0}-\phi^g_X).$$ 
\item Applying the unit and counit natural transformations, and using $\Res^G_{\ast}\circ\iota_{\bA}\simeq\Id$, yields natural transformations of functors  $\Rep(G,\bA)\to\bA$:
$$ {\mathrm{H}}^0(G,-)\;\Rightarrow\; \Res^G_{\ast}\;\Rightarrow \; {\mathrm{H}}_0(G,-).$$
We denote the image of the composite by $\Triv_G: \Rep(G,\bA)\to\bA.$
\end{enumerate}
\end{ddef}

\begin{ex}\label{ExamInv}
In $\Rep_{\mk}G$, the subquotient $\Triv_G(M)$ of $M\in\Rep_{\mk}G$ is isomorphic to the maximal direct summand of $M$ which has trivial $G$-action. 
\end{ex}

\subsubsection{} Consider the set $I=G/H$ of left cosets and pick a representative $r_i\in G$ for each $i\in I$.
For each $g\in G$ and $i\in I$ we then have some $g(i)\in I$ and $h^g_i\in H$ such that $g r_i=r_{g(i)} h^g_i$. We now also assume that for each $X_0\in\bA$ we have a fixed set of isomorphisms 
$$\{\beta_i^{X_0}: X_0\,\stackrel{\sim}{\to}\, X_0^{(i)}\;|\; i\in I\}\quad\mbox{ in $\bA$}.$$

\begin{ddef}
The functor 
$$\Ind_H^G\,:\; \Rep(H,\bA)\to\Rep(G,\bA)$$
maps an object $X$ in $\Rep(H,\bA)$ to $Y=(Y_0,\phi_Y)$ with $Y_0=\bigoplus_{i\in I}X_0^{(i)}$ and
$$\phi_Y^g\;=\; \left( \delta_{i,g(j)}\,\beta^{X_0}_i\circ\phi_X^{h_j^g}\circ (\beta^{X_0}_j)^{-1} \right)_{i,j\in I}.$$
For a morphism $f$ from $X$ to $Z$ in $\Rep(H,\bA)$ we have $\Ind_H^G(f)=\left(\beta^{Z_0}_i\circ f\circ (\beta^{X_0}_i)^{-1}\right)_{i\in I}$. \end{ddef}
As in the classical case, the functor $\Ind^G_H$ is left and right adjoint to $\Res^G_H$.

\subsection{Elementary properties}

For $g\in G$ we denote by $H^g$ the subgroup $gHg^{-1}<G$. Since $H\simeq H^g$ we can interpret $H$-representations as $H^g$-representations. Concretely, for $X\in \Rep(H,\bA)$, we denote by $X^g$ the object in $\Rep(H^g,\bA)$ which has same underlying object in $\bA$, but has action given by $\phi^{ghg^{-1}}_{X^g}=\phi^h_X$.

\begin{lemma}[Mackey's theorem]\label{LemMack}
For a subgroup $L<G$, we have natural isomorphisms
$$\Res^G_L\circ\Ind^G_HX\;\stackrel{\sim}{\to}\; \bigoplus_{s\in L\backslash G/H} \Ind^L_{L\cap H^s}\circ \Res^{H^s}_{L\cap H^s} X^s,\quad\mbox{for $X\in \Rep(H,\bA)$.}$$

\end{lemma}
\begin{proof}
The classical proof, see e.g.~\cite[Lemma~III.8.7]{Alperin}, carries over verbatim.
\end{proof}

\begin{lemma}\label{LemPfff}
For $X$ in $\Rep(H,\bA)$, the morphisms in $\bA$ given by $(\beta_i^{X_0})_{i\in I}:X\to \Ind^G_HX$ and $((\beta_i^{X_0})^{-1})_{i\in I}:\Ind^G_HX\to X$, induce isomorphisms
$$ {\mathrm{H}}^0(H,X)\;\stackrel{\sim}{\to}\;  {\mathrm{H}}^0(G,\Ind^G_HX)\qquad\mbox{and}\qquad  {\mathrm{H}}_0(G,\Ind^G_HX)\;\stackrel{\sim}{\to}\;  {\mathrm{H}}_0(H,X).$$
\end{lemma}
\begin{proof}
We prove the first property, as the second is similar. Take a trivial $G$-representation $Z$ in $\Rep(G,\bA)$, {\it i.e.} an object in the image of $\iota_{\bA}$. A morphism 
$f$ from $Z$ to $\Ind^G_HX$ in $\bA$ is of the form $(f_i)_{i\in I}$ for some $f_i: Z\to X_0^{(i)}$. Then $f\in \Hom_G(Z,\Ind^G_HX)$ if and only if
$$\beta_{g(j)}^{X_0}\circ \phi_X^{h^g_j}\circ (\beta^{X_0}_j)^{-1}\circ f_j\;=\; f_{g(j)},\quad\mbox{for all $j\in I$ and $g\in G$.}$$

Fix an arbitrary $i_0\in I$. The above equation for $j=i_0$ and arbitrary $g\in H^{r_{i_0}}$ implies that $\varphi:=(\beta_{i_0}^{X_0})^{-1}\circ f_{i_0}$ is in $\Hom_H(Z,X)$. The equation for $j=i_0$ and $g=r_{i}r_{i_0}^{-1}$ for all $i\in I$ then shows that
$f_{i}= \beta_{i}^{X_0}\circ  \varphi$ for all $i\in I$.
We have thus showed that composing with $(\beta_i^{X_0})_{i\in I}:X\to \Ind^G_HX$ in $\bA$ induces an epimorphism
$$\Hom_H(Z, X)\;{\tto}\;\Hom_G(Z, \Ind^G_HX).$$
Since we compose with an monomorphism in $\bA$, the above epimorphism is also a monomorphism.
The fact that composition with $(\beta_i^{X_0})_i$ induces an isomorphism for each such $Z$ concludes the proof.
\end{proof}

\begin{cor}\label{CorIndex}
Assume $\bA$ is $\mk$-linear. 
\begin{enumerate}[(i)]
\item If the image of $|G:H|$ in $\mk$ is zero, we have
$\Triv_G\circ\Ind^G_H=0$.
\item If $|G:H|$ is zero and $|G:L|$ is invertible in $\mk$, for
$L<G$, then
$\Triv_L\circ \Res^G_L\circ\Ind^G_H=0$.
\item If $|G:H|$ is invertible in $\mk$, then $\Triv_G\circ\Ind^G_H\simeq \Triv_H$.
\end{enumerate}
\end{cor}
\begin{proof}
Lemma~\ref{LemPfff} implies that there exists a commutative diagram in $\bA$
$$\xymatrix{
 {\mathrm{H}}^0(H,X)\ar[r]^-{\sim}\ar@{^{(}->}[d]&  {\mathrm{H}}^0(G,\Ind^G_HX)\ar@{^{(}->}[dr]&&   {\mathrm{H}}_0(G,\Ind^G_HX)\ar[r]^-{\sim}& {\mathrm{H}}_0(H,X)\\
X\ar@{^{(}->}[rr]&&\Ind^G_HX\ar@{->>}[rr]\ar@{->>}[ru]&&X\ar@{->>}[u],
}$$
such that the composition of the lower horizontal line is $|G:H|$ times $\Id_X$.
The morphism from $ {\mathrm{H}}^0(G, \Ind^G_HX)$ to $ {\mathrm{H}}_0(G,\Ind^G_HX)$ defining $\Triv_G\Ind^G_HX$ is therefore, up to composition with isomorphisms, equal to
 $|G:H|$ times the corresponding morphism from $ {\mathrm{H}}^0(H,X)$ to $ {\mathrm{H}}_0(H,X)$. This proves parts (i) and (iii).

Now we prove part (ii). By Lemma~\ref{LemMack}, the functor $\Res^G_L\circ\Ind^G_H$ is a direct sum of inductions to $L$ from subgroups $L'<L$ which are isomorphic to subgroups of $H$. By assumption and Lagrange's theorem we know that $|L:L'|$ is zero in $\mk$, which implies we can apply part (i) for the group~$L$.
\end{proof}

\begin{lemma}
\label{LemSubQuo}
\begin{enumerate}[(i)]
\item The object $\Triv_GX$ is a subquotient in $\Triv_HX$. 
\item If $H$ is a normal subgroup of $G$, then
$$ {\mathrm{H}}^0(G,-)\;\simeq\; {\mathrm{H}}^0(G/H, {\mathrm{H}}^0(H,-))\quad\mbox{and}\quad  {\mathrm{H}}_0(G,-)\;\simeq\;{\mathrm{H}}_0(G/H, {\mathrm{H}}_0(H,-).$$
\item $\Triv_HX$ is canonically a $G/H$-object and $\Triv_GX$ is a subquotient in $\Triv_{G/H}\Triv_HX$.
\end{enumerate}
\end{lemma}
\begin{proof}
Part (i) follows from the commutative diagram
\begin{equation}\label{eqCD}\xymatrix{
 {\mathrm{H}}^0(H,X)\ar@{->>}[rr]&&\Triv_HX\ar@{^{(}->}[rr]&& {\mathrm{H}}_0(H,X)\ar@{->>}[d]\\
 {\mathrm{H}}^0(G,X)\ar@{^{(}->}[u]\ar[rrrr]&&&& {\mathrm{H}}_0(G,X),
}\end{equation}
where the image of the lower horizontal morphism is $\Triv_GX$. 

Now assume that $H<G$ is normal. Part (ii) is obvious. Clearly the map $ {\mathrm{H}}^0(H,X)\to {\mathrm{H}}_0(H,X)$ is $G/H$-equivariant.
Part (iii) then follows by definition and extending diagram \eqref{eqCD} to include ${\mathrm{H}}^0({G/H}, \Triv_HX)$ and ${\mathrm{H}}_0({G/H}, \Triv_HX)$.
\end{proof}

\begin{ex}
Already for $\bA=\vecc_{\mk}$, the subquotient $\Triv_GX$ will in general not be isomorphic to $\Triv_{G/H}\Triv_HX$, for $H\lhd G$. An example is given by $\mk$ a field of characteristic $2$ and $G=\SG_2\times\SG_2$. Indeed, let $X$ be a $3$-dimensional indecomposable $G$-representation and $H$ one of the copies of $\SG_2$. Then we find $\Triv_GX=0$ but $\Triv_{G/H}\Triv_HX\simeq\mk$.
\end{ex}

Recall the natural automorphisms $\phi^g$ of $\Res^G_\ast$ in equation~\eqref{GFrg}. 

\begin{lemma}\label{DirSumm}Assume that $\bA$ is $\mk$-linear and that $n:=|G:H|$ is invertible in $\mk$.
\begin{enumerate}[(i)]
\item The natural endomorphism $f:=\frac{1}{n}\sum_{i\in I}\phi^{r_i}$ of $\Res^G_\ast$ restricts to $h: {\mathrm{H}}^0(H,-)\Rightarrow  {\mathrm{H}}^0(G,-)$.
\item The natural endomorphism $f':=\frac{1}{n}\sum_{i\in I}\phi^{(r_i^{-1})}$ of $\Res^G_\ast$ yields $h': {\mathrm{H}}_0(G,-)\Rightarrow  {\mathrm{H}}_0(H,-)$.
\item The functor $\Triv_G$ is a direct summand of $\Triv_H$.
\end{enumerate}
\end{lemma}
\begin{proof}
We fix an arbitrary $X$ in $\Rep(G,\bA)$.
First we prove part (i). We define the morphism $m$ in $\bA$ by the commutative diagram
$$\xymatrix{
 {\mathrm{H}}^0(H,\Res^G_HX)\ar@{^{(}->}[rr]\ar[rrd]^m&& \Res^G_\ast X\ar[d]^{f_X}\\
 {\mathrm{H}}^0(G,X)\ar@{^{(}->}[rr]&& \Res^G_\ast X.
}$$
It then follows by direct computation that $\phi^g_X\circ m=m$ for all $g\in G$, which implies that $m$ factors through $ {\mathrm{H}}^0(G, X)$. Part (ii) is proved similarly.

Now we claim that the morphisms $h_X$ and $h'_X$ as defined in parts (i) and (ii), yield a commutative diagram, natural in $X$,
$$\xymatrix{
 {\mathrm{H}}^0(G,X)\ar@{^{(}->}[r]\ar@{^{(}->}[d]\ar@/_3pc/[dd]_{\Id}&X\ar@{->>}[r]\ar[d]^{f'_X}& {\mathrm{H}}_0(G,X)\ar@{^{(}->}[d]^{h'_X}\ar@/^3pc/[dd]^{\Id}\\
 {\mathrm{H}}^0(H,X)\ar@{^{(}->}[r]\ar@{->>}[d]^{h_X}&X\ar@{->>}[r]\ar[d]^{f_X}& {\mathrm{H}}_0(H,X)\ar@{->>}[d]\\
 {\mathrm{H}}^0(G,X)\ar@{^{(}->}[r]&X\ar@{->>}[r]& {\mathrm{H}}_0(G,X),
}$$
where the unlabelled morphisms are from diagram \eqref{eqCD}.
That the left upper square is commutative follows from the observation that $f'_X$ restricts to the identity on $ {\mathrm{H}}^0(G,X)$. The lower left square is commutative by part (i). Furthermore, since $f_X\circ f'_X$ restricts to the identity on $ {\mathrm{H}}^0(G,X)$, the composite of the two morphisms in the left column is the identity, which implies in particular that $h_X$ is an epimorphism. The arguments for the right-hand side of the diagram are identical.

By commutativity, the morphisms in the right column restrict to morphisms between the respective subobjects $\Triv_GX$ and $\Triv_HX$. In particular, $\Triv_HX$ is a retract of $\Triv_GX$.
By naturality, this proves part (iii).
\end{proof}

\begin{lemma}\label{LemIndRes}
If $\bA$ is $\mk$-linear and $|G:H|$ invertible in $\mk$, then the identity functor on $\Rep(G,\bA)$ is a direct summand of $\Ind^G_H\circ\Res^G_H$.
\end{lemma}
\begin{proof}
We have a morphism
$$(\beta_i^{X_0}\circ\phi^{(r_i^{-1})}_X)_{i\in I}\;:\; X\to\Ind^G_H\Res^G_H(X)$$
and a similarly defined morphism in the other direction which compose to $|G:H|$ times the identity.
\end{proof}

The following proposition can be thought of as a very incomplete categorical generalisation of Green's correspondence, see e.g.~\cite[Chapter~III]{Alperin}.
\begin{prop}\label{PropGreen}
Assume that $\bA$ is $\mk$-linear and that $p:=\charr(\mk)>0$. Let $P$ denote a Sylow $p$-subgroup of $G$ and $L=N_G(P)$ its normaliser. If $H$ contains $L$, then
$$\Triv_G\;\simeq\;\Triv_H.$$
More precisely, the canonical morphism ${\mathrm{H}}^0(G,-)\Rightarrow \Triv_H$ is an epimorphism and $\Triv_H\Rightarrow {\mathrm{H}}_0(G,-)$ is a monomorphism.
\end{prop}
\begin{proof}
It suffices to prove the claim for $H=L$, the claim for any intermediate group $L< H< G$ then follows from diagram~\eqref{eqCD}. We will prove that $\mathrm{H}^0(G,X)\to \Triv_L(X)$ is an epimorphism, for every $X$ in $\Rep(G,\bA)$. The dual claim for $\mathrm{H}_0$ is proved identically. By Lemma~\ref{LemIndRes}, it suffices to consider only $X:=\Ind^G_LY$ for $Y\in\Rep(L,\bA)$.
Consider the diagram
$$\xymatrix{
\mathrm{H}^0(G,\Ind^G_LY)\ar@{^{(}->}[r]& \mathrm{H}^0(L,\Res^G_L\Ind^G_L Y)\ar@{->>}[r]\ar[d]&\Triv_L(\Res^G_L\Ind^G_L Y)\ar[d]\\
\mathrm{H}^0(L,Y)\ar[u]^{\sim}\ar@{=}[r]&\mathrm{H}^0(L,Y)\ar@{->>}[r]& \Triv_L(Y),
}$$
where the left isomorphism is given in Lemma~\ref{LemPfff} and the two downwards vertical arrows are induced from the projection onto the direct summand $Y$ in $\Res\Ind Y$ corresponding to the  identity coset in $G/L$. By construction, the right square is commutative and it follows from the definition in Lemma~\ref{LemPfff} that also the left square is commutative. 

By Sylow's theorems, all Sylow subgroups are conjugate.
Since $P\lhd L$, it is the unique Sylow $p$-subgroup of $L$.
By Lemma~\ref{LemMack}, the complement of the above direct summand $Y$ in $\Res^G_L\Ind^G_L Y$ is an object $Z$ in $\Rep(L,\bA)$ which is a direct sum of objects induced from $L^s\cap L$ to $L$, where $s\in G$ is such that $P^s\not=P$. Consequently, $L^s\cap L$ does not contain the Sylow $p$-subgroup of $L$. Corollary~\ref{CorIndex}(i) then implies $\Triv_LZ=0$. In other words, the right 
 downwards arrow is an isomorphism.

The required epimorphism can now be read off from the commutative diagram.
\end{proof}

\begin{lemma}\label{TensorInd}
Assume $\bA$ is $\mk$-linear and take $M\in \Rep_{\mk}G$ and $X\in \Rep(H,\bA)$. We have an isomorphism in $\Rep(G,\bA)$
$$M\otimes \Ind^G_HX\;\stackrel{\sim}{\to}\; \Ind^G_H(M\otimes X).$$
\end{lemma}
\begin{proof}
This follows from the adjunction between $\Ind^G_H$ and $\Res^G_H$ and equation~\eqref{adjM}.\end{proof}

\subsubsection{}
 Consider a symmetric monoidal category $\bC$ as in \ref{SMC}.
 For every $X\in\bC$ and $n\in\mN$ the braiding $\gamma$ defines a group homomorphism $\SG_n\to \Aut(\otimes^nX)$. The permutation $(1,2)$ is for instance sent to $\gamma_{XX}\otimes (\otimes^{n-2}\Id_X)$. We can thus interpret `$\otimes^n$' as a (non-additive) functor
\begin{equation}\label{funtens}X\mapsto \otimes^nX,\; \;\bC\to \Rep(\SG_n,\bC).\end{equation}



\begin{lemma}\label{Lemnp1}
The object $\Triv_{\SG_{n+1}}(\otimes^{n+1}X)$ is a subquotient of $\Triv_{\SG_n}(\otimes^nX)\otimes X$. Consequently, $\Triv_{\SG_n}(\otimes^nX)=0$ implies that $\Triv_{\SG_r}(\otimes^rX)=0$ for all $r\ge n$.
\end{lemma}
\begin{proof}
This is a special case of Lemma~\ref{LemSubQuo}(i), together with the fact that $\Triv_{\SG_n\times 1}(\otimes^{n+1}X)$ is a quotient of $\Triv_{\SG_n}(\otimes^nX)\otimes X$ (where the quotient map is an isomorphism if $X$ is flat).
\end{proof}

\subsection{Semisimplification of representation categories}
In this subsection we assume that~$\mk$ is a {\em splitting field} for $G$. By this we mean that every indecomposable module of $\mk G$ is absolutely indecomposable. Equivalently, the radical of $\End_G(M)$ is of codimension 1, for every indecomposable $\mk G$-module $M$. Every algebraically closed field is thus a splitting field for any finite group. Recall the semisimplifcation $\Rep G\to\overline{\Rep G}, \; X\mapsto \overline{X}$ of \ref{semisimp}.

\begin{lemma}\label{LemRepG} Consider arbitrary indecomposable $M,N$ in $\Rep G$.
\begin{enumerate}[(i)]
\item The object $\overline{M}$ is simple or zero. Set $n_M=0$ when $\overline{M}=0$ and $n_M=1$ otherwise.
\item If $\overline{M}\simeq \overline{N}$ then either $M\simeq N$ or $\overline{M}=0=\overline{N}$.
\item For $\delta_{MN}$ defined by $\delta_{MN}=1$ if $M\simeq N$ and $\delta_{MN}=0$ otherwise, we have
$$\dim_{\mk}\Triv_G(M^\ast\otimes N)\;=\;{ \delta_{MN}n_M}.$$
\item The category $\overline{\Rep G}$ is schurian.
\end{enumerate}
\end{lemma}
\begin{proof}
For the entire proof, let $M,N\in\Rep G$ be indecomposable $\mk G$-modules.
By construction, $\End(\overline{M})$ is a quotient of the local algebra $\End_G(M)$ and thus local or zero. Consequently, $\overline{M}$ is either indecomposable or zero. Since $\oRepG$ is semisimple, part (i) follows.
Part (iv) follows from part (i) and the  assumption that $\mk$ is a splitting field for $G$.

Now assume that $M,N$ are not isomorphic and fix a morphism $f:M\to N$. For any morphism $g:N\to M$ we have that $g\circ f$ is not invertible in $\End_G(M)$. Since $\End_G(M)$ is a local and finite dimensional algebra, $g\circ f$ is thus nilpotent. It follows that the morphism $\overline{g}\circ\overline{f}$ of the simple (or zero) object $\overline{M}$ is nilpotent and hence zero. This proves part (ii).

Now let $M,N$ be arbitrary again. As a special case of part (ii), the only indecomposable module in $\Rep G$ which is mapped to $\unit$ in $\overline{\Rep G}$ is the trivial one. By Example~\ref{ExamInv}, we get isomorphisms of vector spaces
$$\Triv_G(M^\ast\otimes N)\;\simeq\; \mk^{\oplus [\overline{M}^\vee\otimes \overline{N}:\unit]}\;\simeq\; \Hom(\overline{M},\overline{N}).$$
Part (iii) then follows from parts (ii) and (iv).\end{proof}

\begin{rem}
It is easy to see that $\overline{M}=0$ precisely when $ \dim_{\mk}M$ is divisible by $p$. 
\end{rem}

\subsubsection{}  For each isomorphism class of indecomposable modules $M$ in $\Rep G$ with $n_M=1$ (as defined in Lemma~\ref{LemRepG}(i)) we choose one representative.
We denote the corresponding set by $\mB\overline{G}\subset\Ob\Rep G$. We can interpret $\mB\overline{G}$ as the canonical basis of the Grothendieck group of $\oRepG$.

\begin{ddef}\label{defss}
Assume that $\mk G$ is of finite representation type and $\bA$ is $\mk$-linear. We define the {\bf semisimplification functor} 
$$S_G:\;\Rep(G,\bA)\,\to\, \bA\boxtimes\oRepG \quad\mbox{by}\quad X\mapsto \bigoplus_{M\in \mB\overline{G}} \left(\Triv_{G}(M^\ast\otimes X)\boxtimes \overline{M}\right). $$
\end{ddef}

\begin{prop}\label{propsimplification}
Assume that $\bA$ is semisimple and schurian. Then the composite of
$$\bA\boxtimes\Rep G\;\stackrel{\sim}{\to}\;   \Rep(G,\bA)  \;\stackrel{S_G}{\to}\;\bA\boxtimes\oRepG $$
is just the product of the identity functor on $\bA$ and $\Rep G\to \oRepG:\,M\mapsto\overline{M}$.
\end{prop} \begin{proof}
For simplicity we consider an indecomposable module $N\in \Rep G$ and some object $X_0\in \bA$. The composite is then
$$X_0\boxtimes N\;\mapsto\; N\otimes X_0\;\mapsto\;  \bigoplus_{M\in \mB\overline{G}} \left(\Triv_{G}(M^\ast\otimes N)\otimes X_0\right)\boxtimes \overline{M}= X_0\boxtimes \overline{N},$$
by Lemma~\ref{LemRepG}
\end{proof}


\section{Local semisimplicity and freeness }\label{LSSTC}
We fix an arbitrary field $\mk$ for the entire section.
\subsection{Definitions}\label{SecDef}

For this subsection we fix a monoidal category $\bC$ as in \ref{SMC}.

\subsubsection{}\label{alphan} For a morphism $\alpha:\unit\hookrightarrow X$ in $\bC$ and $n\in\mN$, we define 
$$\alpha^n\in\Hom(\unit,\Sym^nX)\quad\mbox{as the composition }\; \unit \xrightarrow{\otimes^n\alpha}\otimes^nX\tto \Sym^nX. $$
We can define $\alpha^n$ also as $\alpha^n={\mathrm{H}}_0({\SG_n},\otimes^n\alpha)$, or via the algebra morphism
$$\alpha^\bullet=\oplus_{n}\alpha^n:\;\Sym^\bullet \unit\;\to\; \Sym^\bullet X.$$

\subsubsection{} Fix a short exact sequence $\Sigma$ in $\bC$
$$\Sigma:\;0\to U\to V\to W\to 0,$$
with $U$ and $V$ flat.
 This filtration of length $2$ on $V$ induces a filtration of length $n+1$ on ${\otimes^n}V$ with $\gr(\otimes^n V)\simeq \otimes^n(\gr V)$. The quotient $\Sym^n V$ of ${\otimes^n}V$ is thus also filtered and we get a canonical graded epimorphism
$$\theta_{\Sigma}^n:\Sym^n(\gr V)\;\tto\; \gr( \Sym^n V).$$ If $2\le \charr(\mk)\le n$, then $\theta_{\Sigma}^n$ need a priori not be an isomorphism, since ${\mathrm{H}}_0(\SG_n,-)$ is then only right exact. 

For a monomorphism $\alpha:\unit\hookrightarrow X$ with $X$ flat, 
the morphism~$\alpha^n$ of \ref{alphan} is the degree 1 component of $\theta_{\Sigma}^n$ composed with the inclusion $\gr_1\Sym^nX\hookrightarrow \Sym^nX$.

\subsection{Locally semisimple categories}\label{SecChara}

\begin{ddef}\label{DefLSS}
A tensor category $\bT$ is {\bf locally semisimple} if there exists a symmetric monoidal category~$\bC$ as in \ref{SMC} and a tensor functor $F:\bT\to\bC$ which maps every short exact sequence $\Sigma$ in $\bT$ to a split short exact sequence $F(\Sigma)$.\end{ddef}

By Lemma~\ref{LemFinLen}(ii), all tensor categories which admit fibre functors in the sense of Definition~\ref{DefFibre} are locally semisimple.
We can characterise locally semisimple tensor categories internally as follows.
We freely use the notation and definitions of Subsection~\ref{SecDef} and the tensor functor $F_{\scA}=\scA\otimes-$ from \eqref{extsca}. Some related results can be found in \cite[Proposition 5.3.4]{Sch2}.
\begin{thm}\label{Thm1}
A tensor category $\bT$ is locally semisimple if and only if one of the following equivalent properties is true.
\begin{enumerate}[(i)]
\item For every short exact sequence $\Sigma$ in $\bT$, the epimorphism $\theta_{\Sigma}$ is an isomorphism.
\item For every $X\in\bT$, $n\in\mN$ and non-zero $\alpha\in\Hom(\unit,X)$, the morphism $\alpha^n$ is non-zero.
\item For every short exact sequence $\Sigma$ in $\bT$ there exists a non-zero $\scA=\scA_\Sigma$ in $\Alg{\bT}$ such that $\scA\otimes \Sigma$ splits in $\Mod_{\scA}$.
\item There exists non-zero $\scA\in \Alg{\bT}$ such that for every short exact sequence $\Sigma$ in $\bT$, the sequence $\scA\otimes\Sigma$ splits  in $\Mod_{\scA}$.
\end{enumerate}
\end{thm}

\begin{rem}\label{RemLSS}\begin{enumerate}[(i)]
\item If $\charr(\mk)=0$, Theorem~\ref{Thm1}(i) shows that all tensor categories are locally semisimple, see also \cite[Lemme~7.14]{Del90}. If $\charr(\mk)>0$, we will improve Theorem~\ref{Thm1}(ii) to Theorem~\ref{Thm1p}.
\item For $p=2$, an example of a tensor category which does not satisfy the properties in Theorem~\ref{Thm1} is given in \cite[Example~3.3]{EHO}, see also~\cite{BE}.
For $p>2$, such tensor categories will appear in future work of Benson, Etingof and Ostrik and of the author.
\item If $\charr(\mk)\not=2$, just as in the proof of Theorem~\ref{Thm1}(i), we can show that the canonical monomorphism $\gr( \Gamma_{(1^n)}X)\hookrightarrow \Gamma_{(1^n)} (\gr X)$ is always an isomorphism for a filtered object $X$ in a locally semisimple tensor category. The theorem thus shows that in case $\theta$ is always an isomorphism, so is `$\phi_-$' in \cite[Question~3.5]{EHO}. 
\end{enumerate}
\end{rem}

\begin{thm}\label{Thm1p}
A tensor category $\bT$ over a field $\mk$ with $p:=\charr(\mk)>0$ is locally semisimple if and only if for each non-zero $\alpha:\unit\to X$ in $\bT$, the morphism $\alpha^p:\unit\to\Sym^pX$ is non-zero.
\end{thm}

 We fix a tensor category $\bT$ and start the proof of the theorems with some preparatory results. The following lemma is essentially a reformulation of \cite[Exemple~7.12]{Del90}.

\begin{lemma}\label{CommSplit}
Consider a short exact sequence
$$\Sigma:\; 0\to \unit \stackrel{\alpha}{\to} X\to Y\to 0$$
in $\bT$. For $(\scA,m,\eta)\in \Alg{\bT}$, the sequence $\scA\otimes\Sigma$ splits in $\Mod_{\scA}$ if and only if we have an algebra morphism $\Sym^\bullet X\to \scA$ yielding a commutative diagram of algebra morphisms
$$\xymatrix{
\Sym^\bullet X\ar[rr]&& \scA\\
\Sym^\bullet \unit\ar[rr]^\rho\ar[u]^{\alpha^\bullet}&&\unit\ar[u]^{\eta}
}$$
where $\rho$ restricts to the identity $\Sym^1\unit\,\stackrel{=}{\to}\,\unit$ in degree $1$.
\end{lemma}
\begin{proof}
For any algebra $\scA$ we have a commutative diagram
$$\xymatrix{
\Hom_{\scA}({\scA}\otimes X,{\scA})\ar[r]^-\sim\ar[d]^{-\circ(\Id_{\scA}\otimes \alpha)}&\Hom( X,{\scA})\ar[r]^-\sim\ar[d]^{-\circ\alpha}&\Hom_{alg}(\Sym^\bullet X,{\scA})\ar[d]^{-\circ\alpha^\bullet}\\
\Hom_{\scA}({\scA},{\scA})\ar[r]^-\sim&\Hom( \unit,{\scA})\ar[r]^-\sim&\Hom_{alg}(\Sym^\bullet \unit,{\scA}),
}$$
see \cite[Example~7.9]{Del90}.
A morphism
$f\in \Hom_{\scA}(\scA\otimes X,\scA)$ splits $\scA\otimes\Sigma$
if and only if $(\Id_{\scA}\otimes \alpha)\circ f=\Id_{\scA}$. 
With $g\in\Hom_{alg}(\Sym^\bullet X,{\scA})$ the image of $f$ under the isomorphisms, this condition becomes commutativity of the diagram
$$\xymatrix{
\Sym^\bullet X\ar[rr]^g&& {\scA},\\
\Sym^\bullet \unit\ar[u]^{\alpha^\bullet}\ar[urr]_{\eta\circ\rho}
}$$
which concludes the proof.
\end{proof}

\begin{cor}\label{CorLS}
If for a short exact sequence $\Sigma$ as in Lemma~\ref{CommSplit} we have $\alpha^n\not=0$ for all $n\in\mN$, there exists non-zero $\scA\in \Alg{\bT}$ such that $\scA\otimes\Sigma$ splits in $\Mod_{\scA}$.
\end{cor}
\begin{proof}
By Lemma~\ref{CommSplit} it suffices to prove that the pushout in $\Alg{\bT}$
$$\Sym^\bullet X\sqcup_{\Sym^\bullet \unit}\unit\;\simeq\; \Sym^\bullet X\otimes_{\Sym^\bullet \unit}\unit=:\scB$$ is non-zero. 
By construction, in $\Ind\bT$ we have $\scB=\varinjlim \Sym^nX$, where the morphisms in the direct system are given by the composites
$$\Sym^nX\;\stackrel{\Id\otimes\alpha}{\hookrightarrow}\; (\Sym^nX)\otimes X\; \tto\;\Sym^{n+1}X.$$ 
Consequently, the collection of monomorphisms $\{\alpha^n:\unit\to\Sym^nX\}$ yields a monomorphism $\unit\hookrightarrow\scB$, which proves that the pushout is non-zero.
\end{proof}

\begin{proof}[Proof of Theorem~\ref{Thm1}]
Assume we have $F:\bT\to\bC$ as in Definition~\ref{DefLSS}. We then have $F(\theta_\Sigma^n)=\theta_{F^n(\Sigma)}$. Since $F(\Sigma)$ splits, clearly $\theta_{F(\Sigma)}$ is an isomorphism. Since $F$ is faithful, see Lemma~\ref{LemFaith}, it follows that $\theta_{\Sigma}$ is an isomorphism as well. Hence a locally semisimple tensor category satisfies~(i).

Property (i) contains (ii) as a special case. That (ii) implies (iii) follows from Corollary~\ref{CorLS} and the isomorphism between $\Ext^1(X,Y)$ and $\Ext^{1}(Y^\vee\otimes X,\unit)$, for $X,Y\in\bT$, see e.g.~\cite[proof of Lemme~7.14]{Del90}.

If (iii) is true, then for every short exact sequence $\Sigma$ in $\bT$ we have an algebra $\scA_{\Sigma}$ in $\Alg{\bT}$ which splits $\Sigma$. Since $\bT$ is essentially small we can take a set $T$ of short exact sequences such that every short exact sequence in $\bT$ is isomorphic to one in $T$. Then 
$$\scA\;=\;\bigotimes_{\Sigma\in T}\scA_\Sigma\;:=\; \varinjlim_{S}\bigotimes_{\Sigma\in S} \scA_\Sigma\;\in\;\Alg\bT,$$
where $S$ ranges over all finite subsets of $T$, satisfies condition (iv).

If (iv) is satisfied, the tensor functor
$F_{\scA}=\scA\otimes -$ from \eqref{extsca}
 makes $\bT$ locally semisimple.
\end{proof}

\begin{proof}[Proof of Theorem~\ref{Thm1p}]
One direction is a special case of Theorem~\ref{Thm1}(ii). Now assume that $\alpha^p$ is never zero for non-zero $\alpha$ and pick one such $\alpha:\unit\to X$. By iterating $j$ times, we find that the morphism
$$\unit\to \Sym^p(\Sym^p(\cdots \Sym^p(X)\cdots))$$
is non-zero. By iteration of Lemma~\ref{LemSubQuo}(ii), the above morphism can be written as ${\mathrm{H}}_0({Q_j},\otimes^{p^j}\alpha)$, for $Q_j<\SG_{p^j}$ as in \ref{pgroups}. Since $\unit=\otimes^{p^j}\unit$ is in particular $Q_j$-invariant, we actually find that $\Triv_{Q_j}(\otimes^{p^j}\alpha)\not=0$. By Proposition~\ref{PropGreen} and Lemma~\ref{LemNormSyl}, we thus find that $\Triv_{\SG_{p^j}}(\otimes^{p^j}\alpha)\not=0$, so in particular $\alpha^{p^j}= {\mathrm{H}}_0({\SG_{p^j}},\otimes^{p^j}\alpha)\not=0$, for all $j\in\mN$. Since $\alpha^n=0$ implies $\alpha^{n+1}=0$, we thus find that $\alpha^n\not=0$ for all $n\in\mN$. The conclusion now follows from Theorem~\ref{Thm1}(ii).
\end{proof}

\subsection{Locally free objects and splitting algebras} For the entire subsection, we consider tensor categories $\bT$ and $\bV$, with $\bV$ schurian and semisimple. 
\begin{ddef}\label{DefLocFree}
An object $X\in \bT$ is {\bf locally $\bV$-free} if there exist a non-zero $\scA\in\Alg(\bT\boxtimes\bV)$ and $X_0\in\bV$ such that $\scA\otimes X\simeq \scA\otimes X_0$ in $\Mod_{\scA}^{\bT\boxtimes\bV}$.
\end{ddef}

\begin{lemma}\label{FreeTensorN} 
\begin{enumerate}[(i)]
\item If $X,Y\in\bT$ are locally $\bV$-free, then so are $X\oplus Y$, $X\otimes Y$ and $X^\vee$.
\item If $\bV$ is pointed, then the locally $\bV$-free objects form a tensor subcategory of $\bT$.
\item If $\bT$ is locally semisimple, any extension of two locally $\bV$-free objects is again locally $\bV$-free.
\item If $\bV$ is pointed and $\bT$ is locally semisimple, the tensor subcategory of $\bT$ in (ii) is a Serre subcategory.
\end{enumerate}
\end{lemma}
\begin{proof}
The first observation is straightforward. To prove (ii) it thus suffices to show that any subquotient of a locally $\bV$-free object is again locally $\bV$-free. Consider a locally $\bV$-free object $X\in\bT$. By Lemma~\ref{LemMaxId} we may assume that there exists $\scA\in \Alg(\bT\boxtimes\bV)$, simple in $\Mod_{\scA}$, and $X_0\in\bV$ such that $\scA\otimes X\simeq\scA\otimes X_0$. Since every simple object $S\in\bV$ satisfies $S\otimes S^\vee\simeq\unit$, it follows that $\scA\otimes S$ is also simple in $\Mod_{\scA}$. Consequently $\scA\otimes X$ is semisimple and for any subquotient $Y$ of $X$, the subquotient $\scA\otimes Y$ of $\scA\otimes X_0$ must be of the form $\scA\otimes Y_0$ for some $Y_0\subset X_0$.

To prove (iii) consider a short exact sequence $\Sigma:$ $0\to X\to E\to Y\to 0$ in $\bT$ where $X$ and $Y$ are locally $\bV$-free. By Theorem~\ref{Thm1}(iii) and assumption, there exists an algebra $\scA_\Sigma\otimes \scA_X\otimes \scA_Y$ which ensures $E$ is locally $\bV$-free. Claim (iv) is just the combination of (ii) and (iii).
\end{proof}


\begin{ddef}\label{DefSA1}
A {\bf $\bV$-splitting algebra $\scA$ for $\bT$} is a non-zero $\scA\in \Alg(\bT\boxtimes\bV)$ such that $\scA$ splits every short exact sequence in $\bT$ (or equivalently in $\bT\boxtimes\bV$) and for every $X\in \bT$ (or equivalently, for all $X\in\bT\boxtimes\bV$) there exists $X_0\in\bV$ for which $\scA\otimes X\simeq \scA\otimes X_0$ in $\Mod_{\scA}$.
\end{ddef}

When $\bV=\vecc$, we say `splitting algebra', rather than `$\vecc$-splitting algebra'.
Recall from Lemma~\ref{LemLax} that, for $\scA\in\Alg(\bT\boxtimes\bV)$, the object $\Gamma_{\bV}\scA$ is an algebra in $\Ind\bV$ and that $\Gamma_{\bV}$ interpreted as a functor $\Mod_{\scA}^{\bT\boxtimes\bV}\to\Mod_{\scR}^{\bV}$ is canonically lax monoidal, for an algebra isomorphism $\scR\stackrel{\sim}{\to}\Gamma_{\bV}\scA$.

\begin{prop}\label{ThmGen2} Fix $\scR\in \Alg\bV$. Consider the full subcategory of the category of $\scR$-algebras $a:\scR\to\scA$ in $\Ind\bT\boxtimes\bV$ (the coslice category $(\scR\downarrow \Alg\bT\boxtimes\bV)$) for which $\scA$ is $\bV$-splitting and $\Gamma_{\bV}(a)$ is an isomorphism.
There is a fully faithful functor from this category to the category of tensor functors $\bT\to\Mod_{\scR}^{\bV}$, given by
$$\Phi:\,\scA\,\mapsto\,\Phi_{\scA}:=\Gamma_{\bV}\circ (\scA\otimes-).$$
\end{prop}
\begin{proof}
For an arbitrary algebra $\scA$ over $\scR$, we can define a left exact and lax monoidal functor $\Phi_{\scA}:=\Gamma_{\bV}\circ (\scA\otimes-)$ from $\bT$ to $\Mod_{\scR}^{\bV}$. In case $\scA$ satisfies the further conditions in the proposition, then $\Phi_{\scA}$ is monoidal, as follows from Lemma~\ref{LemLax}(iii), and exact since $\Gamma_{\bV}$ is additive. 
 Hence in this case $\Phi_{\scA}$ is a tensor functor, which means that $\Phi$ is well-defined.

There exists a commutative diagram (up to isomorphism) of functors
$$\xymatrix{\Ind(\bT\boxtimes\bV)\ar[rrr]^{M\,\mapsto\, \Gamma_{\bV}(M\otimes-)}\ar[rrrd]&&&\Funct_{\mk}(\bT,\Ind\bV)\ar[d]^{F\mapsto \{X\boxtimes V\,\mapsto\,\Hom(V,FX^\vee)\}}\\
&&&\Funct_{\mk}((\bT\boxtimes\bV)^{\op},\Vecc),
}$$
where the diagonal arrow is the ordinary (fully faithful) embedding of ind-objects in the category of all $\mk$-linear presheaves. It is easy to see that the downwards functor is an equivalence, so the horizontal functor is fully faithful.

In particular, for $\scR$-algebras algebras $a:\scR\to\scA$ and $b:\scR\to\scB$ in $\Alg(\bT\boxtimes\bV)$, we have an isomorphism
 $$\Hom(\scA,\scB)\,\stackrel{\sim}{\to}\,\Nat(\Gamma_{\bV}(\scA\otimes-),\Gamma_{\bV}(\scB\otimes-)),\;\,f\mapsto\psi^f=\Gamma_{\bV}(f\otimes-). $$
 To conclude the proof it suffices to show that $f:\scA\to\scB$ is an $\scR$-algebra morphism if and only if $\psi^f$ is a natural transformation of lax monoidal functors to $\Mod_{\scR}$.  Indeed, it then suffices to apply this to the full subcategory of $(\scR\downarrow \Alg\bT\boxtimes\bV)$ in the proposition.

  It follows easily (and similarly to the next paragraph) that $\psi^f$ is a natural transformation of functors to $\Mod_{\scR}$ if and only if $f$ is a morphism of $\scR$-modules (for the canonical $\scR$-module structure on $\scA$ and $\scB$). From now on we only consider such $f$.
 
 If $f:\scA\to\scB$ is an algebra morphism it follows immediately that $\psi^f$ respects the lax monoidal structures of $\Phi_{\scA}$ and $\Phi_{\scB}$. Now we prove the reverse implication. If $f\circ a\not= b$ (which is equivalent to $\Gamma_{\bV}(f\circ a)\not=\Gamma_{\bV}(b)$), then $\psi^f$ will not yield the required commutative diagram for $\Phi_{\scA}(\unit)\leftarrow \scR\to\Phi_{\scB}(\unit)$. If $f\circ m_{\scA}\not=m_{\scB}\circ (f\otimes f)$, then there exist $U,V\in\bV$ and $X,Y\in\bT$ with morphisms $X^\vee\boxtimes U\to\scA$ and $Y^\vee\boxtimes V\to\scA$, such that the corresponding compositions 
$$(X^\vee\boxtimes U)\otimes (Y^\vee\boxtimes V)\;\to \scA\otimes\scA\;\rightrightarrows\;\scB $$
are not equal. It then follows easily that application of $\psi^f$ leads to a non-commutative diagram
$$\Phi_{\scA}(X)\otimes_{\scR}\Phi_{\scA}(Y)\;\rightrightarrows\;\Phi_{\scB}(X\otimes Y)$$
and hence $\psi^f$ is not a morphism of lax monoidal functors. 
This completes the proof.
\end{proof}

\begin{prop}\label{PropApp} 
Assume that every object in $\bT$ is locally $\bV$-free and $\bV$ satisfies  \eqref{CondVer}.
\begin{enumerate}[(i)]
\item The tensor categories  $\bT$ and $\bT\boxtimes\bV$ are locally semisimple. 
\item There exists a $\bV$-splitting algebra for $\bT$.
\item The tensor category $\bT$ admits a fibre functor $\bT\to\Mod_{\scR}^{\bV}$ over some algebra $\scR\in\Alg\bV$.
\end{enumerate}
\end{prop}
\begin{proof}
For (i) it suffices to prove that $\bT':=\bT\boxtimes\bV$ is locally semisimple.
As every object in $\bT'$ is locally free, by taking an (infinite) tensor product over the set of isomorphism classes of objects, we obtain $\scB\in\Alg\bT'$ such that for every $X\in \bT'$ there exists $X_0\in\bV$ such that $\scB\otimes X\simeq\scB\otimes X_0$.
Now take a monomorphism $\alpha:\unit\hookrightarrow X$ in ${\bT'}$. We will show that $\alpha^n\not=0$ for all $n\ge1$, so the conclusion in (i) will follow from Theorem~\ref{Thm1}(ii).

Observe that any tensor functor $F:{\bT'}\to?$ is faithful and satisfies $F(\alpha)^n=F(\alpha^n)$. In particular, we have $\alpha^n\not=0$ if and only if $(\scB\otimes \alpha)^n\not=0$. We compose $\scB\otimes\alpha$ with an isomorphism between $\scB\otimes X$ and $\scB\otimes X_0$ for some $X_0\in\bV$, which exists by assumption, to get a monomorphism
$$\alpha_0:\;\scB\hookrightarrow \scB\otimes X_0\quad\mbox{in $\Mod_{\scB}$.}$$
We must show that $\alpha_0^n\not=0$. Although this can be shown in general, for convenience we will replace $\scB$ by a simple quotient, as we can do by Lemmata~\ref{LemMaxId} and~\ref{RemTriv}.

We claim that the simplicity of $\scB$ implies that $\Hom_{\scB}(\scB,\scB\otimes S)=0$ for simple $S\in\bV$ when $S\not\simeq\unit$. By duality we can equivalently consider a non-zero morphism $\scB\otimes S\to\scB$ for such $S$, which is automatically an epimorphism. Since $-\otimes_{\scB}-$ is right exact this yields an epimorphism between the respective $n$-th tensor powers, for all $n\in\mN$. Since $\otimes^n_{\scB}\scB=\scB=\Sym^n_{\scB}\scB$ and ${\mathrm{H}}_0(\SG_n,-)$ is right exact, we thus find epimorphisms
$$\scB\otimes\Sym^nS\,\simeq\,\Sym^n_{\scB}(\scB\otimes S)\;\tto\;\scB,\quad\mbox{for all $n\in\mN$.}$$
This is contradicted by assumption \eqref{CondVer}.
Hence the monomorphism $\alpha_0$ is the embedding of the simple direct summand $\scB$ of $\scB\otimes X_0$. It follows that $\alpha_0^n\not=0$.
 


By part (i) we have an algebra $\scA\in\Alg(\bT\boxtimes\bV)$ as in Theorem~\ref{Thm1}(iv). The tensor product $\scA\otimes\scB$ is $\bV$-splitting for $\bT$, proving (ii). Furthermore, (ii) implies (iii) by Proposition~\ref{ThmGen2}.
\end{proof}

\subsection{Neutrality}
As in the previous subsection we consider tensor categories $\bT$ and $\bV$, with $\bV$ schurian and semisimple.
\begin{ddef}\label{defneutsplit}
A {\bf neutral $\bV$-splitting algebra $\scA$ for $\bT$} is an algebra $(\scA,m,\eta)$ in $\Alg(\bT\boxtimes\bV)$ for which $\Gamma_{\bV}(\eta)$ is an isomorphism and such that for each $X\in\bT$ there exists $X_0\in\bV$ for which $\scA\otimes X\simeq \scA\otimes X_0$.
\end{ddef}

By the following lemma, this definition is consistent with Definition~\ref{DefSA1}.
 \begin{lemma}
 If $\scA$ is a neutral $\bV$-splitting algebra, then
 \begin{enumerate}[(i)]
 \item $\scA$ is a $\bV$-splitting algebra;
 \item we have a symmetric monoidal equivalence $\Mod_{\scA}^{\bT\boxtimes\bV}\simeq \Ind\bV$.
 \end{enumerate}
 \end{lemma}
 \begin{proof}
 We start by proving (ii), via the symmetric monoidal functor
 $$F:\;\Ind\bV\;\stackrel{\scA\otimes-}{\to}\;\Mod_{\scA}^{\bT\boxtimes\bV}.$$
 It follows easily from $\Gamma_{\bV}(\scA)=\unit$ that $F$ is fully faithful.
 As for any algebra in $\bT\boxtimes\bV$, every $\scA$-module is a quotient of a direct sum of `free modules' $\scA\otimes Z$ with $Z\in\bT\boxtimes\bV$. By the assumptions in \ref{defneutsplit}, any object in $\Mod_{\scA}$ thus has a presentation by objects in the image of $F$. Since $F$ is fully faithful and right exact, it follows it is dense and therefore an equivalence.
 
 By (ii), the target of $\scA\otimes-:\bT\to\Mod_{\scA}$ is semisimple, so clearly the functor $\scA\otimes-$ splits every short exact sequence. This proves (i). \end{proof}

\begin{thm}\label{ThmFibreNeut}
We have an equivalence from the category of neutral $\bV$-splitting algebras $\scA$ for $\bT$ with the category of tensor functors $\bT\to\bV$, given by
$$\scA\,\mapsto\,\Phi_{\scA}:=\Gamma_{\bV}\circ (\scA\otimes-).$$
\end{thm}
\begin{proof}
By Proposition~\ref{ThmGen2} it suffices to prove that the assignment yields a dense functor. 
Consider a  tensor functor $F:\bT\to\bV$. By Lemma~\ref{LemTannForm}(ii), we can interpret $\bT\boxtimes\bV$ as the category of representations of an affine group scheme $G$ in $\bV$, represented by some $\cO\in\Alg\bV$. Furthermore, the functor $\omega:\bT\boxtimes\bV\to\bV$, induced from $(F,\Id_{\bV})$ is to be interpreted as the functor forgetting the $G$-action

Denote by $\underline{\cO}$ the right regular $\cO$-comodule, which is an object in $\Alg(\bT\boxtimes\bV)$. It follows from the standard properties of Hopf algebras that $\underline{\cO}$ is a neutral $\bV$-splitting algebra for $\bT$. We sketch a proof below.

For any $M\in \bT\boxtimes\bV$, composition with the counit of $\cO$ yields a natural isomorphism 
\begin{equation}\label{epsiloniso}\Hom(M,\underline{\cO})\;\simeq \;\Hom(\omega(M),\unit).
\end{equation} 
We will freely use that the canonical inclusion $\iota:\bV\hookrightarrow \bT\boxtimes\bV$ satisfies $\omega\circ\iota\simeq\Id_{\bV}$.
Equation~\eqref{epsiloniso} thus implies in particular
that  $\Gamma_{\bV}(\underline{\cO})=\unit$. Furthermore, since $\bV$ is semisimple, \eqref{epsiloniso} shows that $\underline{\cO}$ is injective in $\Ind(\bT\boxtimes\bV)$. Therefore $\underline{\cO}\otimes-$ splits every short exact sequence in $\bT\boxtimes\bV$. Finally,  \eqref{epsiloniso} shows that $\underline{\cO}\otimes X\simeq \underline{\cO}\otimes \iota\omega(X)$, for any $X\in\bT$.

In conclusion, the tensor functor $F$ is isomorphic to $\Phi_{\underline{\cO}}=\Gamma_{\bV}\circ(\underline{\cO}\otimes-)$.
 \end{proof} 
 
 The proof of Theorem~\ref{ThmFibreNeut} and Lemma~\ref{LemTannForm} yield the following examples.
 \begin{ex}\label{examkGsplits}
 \begin{enumerate}[(i)]
 \item  The neutral splitting algebras for $\bT$ are the algebras in $\Alg\bT$ isomorphic to $\underline{\mk[G]}$, under an equivalence $\bT\simeq\Rep G$ with $G/\mk$ an affine group scheme.
 \item If $p\not=2$, the neutral $\svec$-splitting algebras for $\bT$ are the algebras in $\Alg(\bT\boxtimes\svec)$ isomorphic to $\underline{\mk[G]}$, under an equivalence $\bT\boxtimes\svec\simeq\Rep G$ with $G$ an affine supergroup scheme.
 \end{enumerate}
 \end{ex}

 
  \begin{cor}\label{CorResA}
Let $\scA$ be a neutral $\bV$-splitting algebra for $\bT$, and $\bT^0\subset\bT$ a tensor subcategory.
\begin{enumerate}[(i)]
\item The algebra $\Gamma_{\bT^0\boxtimes\bV}\scA$ is a neutral $\bV$-splitting algebra for $\bT^0$.
\item For a neutral $\bV$-splitting algebra $\scA^0$ for $\bT^0$, we have  $\Phi_{\scA}|_{\bT^0}\simeq \Phi_{\scA^0}$ as tensor functors
if and only if there exists an algebra morphism $\scA^0\to\scA$.
 \end{enumerate}
 \end{cor}
 \begin{proof}
Lemma~\ref{LemLax}(ii) shows that (i) is true and that $\Phi_{\scA}|_{\bT^0}\simeq \Phi_{\Gamma_{\bT^0\boxtimes\bV}(\scA)}$.
 
  Theorem~\ref{ThmFibreNeut} and the above paragraph show that $\Phi_{\scA}|_{\bT^0}\simeq \Phi_{\scA^0}$ if and only if $\scA^0\simeq\Gamma_{\bT^0\boxtimes\bV}\scA$. The fact that the category of tensor functors $\bT^0\to\bV$ is a groupoid, see~\ref{Defeta}, together with Theorem~\ref{ThmFibreNeut}, shows that $\scA^0\simeq\Gamma_{\bT^0\boxtimes\bV}\scA$ if and only if there exists an algebra morphism $\scA^0\to\Gamma_{\bT^0\boxtimes\bV}\scA$, which is the same as an algebra morphism $\scA^0\to\scA$.
 \end{proof}

\section{Frobenius twists in tensor categories}\label{SecFrob}
Consider an arbitrary field with $p:=\charr(\mk)>0$ and a monoidal category $\bC$ as in \ref{SMC}.


\subsection{The symmetric twist}
\subsubsection{} Recall the functor $\otimes^{p^j}:\bC\to\Rep(\SG_{p^j},\bC )$, $X\mapsto \otimes^{p^j}X$ from \eqref{funtens}.


\begin{ddef}
The $j$-th {\bf symmetric Frobenius twist} is the functor
$$\Fr^{(j)}_{+}=\Triv_{\SG_{p^j}}\circ \otimes^{p^j}:\;\bC\,\to\, \bC .$$
\end{ddef}
We also write $\Fr_+:=\Fr_+^{(1)}$. For an exact (for instance out of a tensor category) tensor functor $F:\bC\to\bD$, the diagram
\begin{equation}
\label{eqCommD}\xymatrix{\bC\ar[rr]^{\Fr_{+}^{(j)}}\ar[d]^F&&\bC \ar[d]^F\\
\bD\ar[rr]^{\Fr_{+}^{(j)}}&&\bD
}\end{equation}
commutes up to natural isomorphism. 

\begin{prop}\label{CoolThm}
For a tensor category $\bT$, the following are equivalent:
\begin{enumerate}[(i)]
\item The tensor category $\bT$ is locally semisimple.
\item The functor $\Fr_+:\bT\to \bT $ is exact.
\item The functor $\Fr_+^{(j)}:\bT\to \bT  $ is exact for every $j\in\mN$.
\end{enumerate}
\end{prop}

Before proving the proposition, we return to the more general case of monoidal categories $\bC$ as in \ref{SMC} and prove that $\Fr_+^{(j)}$ is always additive. It can even be made $\mk$-linear by adjusting the $\mk$-linear structure on the target category, but we will omit this interpretation.
\begin{lemma}\label{AddFrob}
The functor $\Fr_{+}^{(j)}$ is additive. So, for all $X,Y\in\bC$, we have
$$\Fr_{+}^{(j)}(X\oplus Y)\;\simeq\; \Fr_{+}^{(j)}(X)\oplus \Fr_{+}^{(j)}(Y),\quad\mbox{for $j\in\mN$}.$$
\end{lemma}
\begin{proof}
For $f,g\in \Hom(X,Y)$ with $X,Y\in\bC$ and $n\in\mN$, we have
$$\otimes^n(f+g)\;=\;\sum_{a+b=n} \Ind^{\SG_n}_{\SG_a\times \SG_b}((\otimes^a f)\otimes (\otimes^b g)).$$

For $n=p^j$, the index $|\SG_n:\SG_a\times\SG_b|$ is given by the binomial coefficient ${p^j\choose a}$. If $a\not\in\{0,p^j\}$, this index ${p^j\choose a}=\frac{p^j}{a}{p^j-1 \choose a-1}$ is divisible by $p$.
By Corollary~\ref{CorIndex}(i),  we thus have
$$\Triv_{\SG_{p^j}}(\otimes^{p^j}(f+g))=\Triv_{\SG_{p^j}}(\otimes^{p^j}f)+\Triv_{\SG_{p^j}}(\otimes^{p^j} g),$$
which demonstrates that the functor is additive.
\end{proof}

\begin{proof}[Proof of Proposition~\ref{CoolThm}]
Assume first that $\bT$ is locally semisimple via the tensor functor $F:\bT\to\bC$. 
By Lemma~\ref{AddFrob} and the assumption that $F$ map every short exact sequence to a split one, the composition $\Fr_{+}^{(j)}\circ F$ is exact. Hence $F\circ \Fr_{+}^{(j)}$ is exact by commutativity of \eqref{eqCommD}. Since $F$ is exact and faithful, the functor $\Fr_{+}^{(j)}:\bT\to\bT  $ is also exact. This proves that (i) implies (iii). Furthermore, property (iii) includes (ii) as a special case.

Now consider a monomorphism $\alpha:\unit\to X$ in $\bT$. We observe that $\alpha^p$ as defined in \ref{alphan} is given by $\Triv_{\SG_p}(\otimes^p\alpha)=\Fr_+(\alpha)$ composed with the monomorphism $\Triv_{\SG_p}(\otimes^pX)\hookrightarrow \Sym^pX$.
Now if $\Fr_{+}$ is exact, then $\Fr_+(\alpha):\unit\to \Fr_+(X)$ is a monomorphism and thus not zero. Consequently $\alpha^p$ is not zero and we apply Theorem~\ref{Thm1p} to show that (ii) implies (i).
\end{proof}

\begin{ex}\label{ExVec} Assume that $\mk$ is perfect.
Take $V\in \vecc$, consider the corresponding algebraic group $\GL(V)$ and the category of algebraic representations $\bT:=\Rep_{\mk}\GL(V)$. 
We have that $\Gamma^nV$, respectively $\Sym^nV$, is isomorphic to the Weyl module $V(n\epsilon_1)$, respectively dual Weyl module ${\mathrm{H}}^0(n\epsilon_1)$, see \cite[\S II.2.16]{Jantzen}. It follows from \cite[Propsition~II.4.13]{Jantzen} that the image of a nonzero morphism from $V(n\epsilon_1)$ to ${\mathrm{H}}^0(n\epsilon_1)$ is the simple module of highest weight $n\epsilon_1$. By \cite[Corollary~II.3.17]{Jantzen} we thus find  $\Fr_+^{(j)}V\simeq V^{(j)}$, where $V^{(j)}$ is the classical $j$-th Frobenius twist of $V$ in $\Rep\GL(V)$.
\end{ex}

From Lemma~\ref{AddFrob} and equation~\eqref{GammaVer} we find the following examples. The first example demonstrates in particular that $\Fr_+$, which is the image of a natural transformation from a lax monoidal to an oplax monoidal functor, is not a monoidal functor when $p\not=2$.
\begin{ex}\label{ExamTwist}
\begin{enumerate}[(i)]
\item If $p\not=2$, for $\bT=\svec$ and $V=V_{\oa}\oplus V_{\ob}\in \svec$, we have $\Fr_+V\simeq V_{\oa}$. 
\item More generally, for $X$ in $\ver_p$, we have $\Fr_+X\simeq \unit^{\oplus [X:\unit]}$.
\item If $p=2$, let $D$ be the triangular Hopf algebra of \cite[Example~3.3]{EHO} and $\bT$ the category of finite dimensional $D$-modules. Then $\Fr_+D=0$, for $D$ the regular $D$-module.
\item The condition that $\mk$ be perfect is required in Example~\ref{ExVec}. For instance, if $\mk=\mF_p(t)$ and $V=\mk$, then $\dim_{\mk}V^{(1)}=p$ but $\dim_\mk\Fr_+V=1$.
\end{enumerate}
\end{ex}

For $j\in\mZ_{>0}$, we denote by $\Fr_+^j:\bC\to\bC$ the composition $\Fr_+\circ\Fr_+\circ\cdots\circ\Fr_+$ of $j$ factors $\Fr_+$.

\begin{lemma}\label{Lempowers}
For all $X\in\bC$, the object $\Fr_+^{(j)}(X)$ is a subquotient of $\Fr_+^j(X)$.
\end{lemma}
\begin{proof}
By Lemma~\ref{LemNormSyl} and Proposition~\ref{PropGreen}, we have
$\Fr_+^{(j)} X\simeq\Triv_{Q_j}(\otimes^{p^j}X),$
with $Q_j<\SG_{p^j}$ introduced in \ref{pgroups}. 
The lemma thus follows by iteration of Lemma~\ref{LemSubQuo}(iii).
\end{proof}

For the rest of the subsection, fix a tensor category $\bT$ and $X,Y\in\bT$.

\begin{lemma}\label{Lemj} The object $\Fr_+(X)\otimes \Fr_+(Y)$ is a subquotient of $\Fr_+(X\otimes Y)$ in $\bT$.
\end{lemma}
\begin{proof}
We have
$$\Fr_+(X)\otimes \Fr_+(Y)\simeq \Triv_{ \SG_{p}\times \SG_{p}}((\otimes^{p}X)\otimes (\otimes^{p}Y))\quad\mbox{and}\quad \Fr_+(X\otimes Y)\simeq \Triv_{ \SG_{p}}((\otimes^{p}X)\otimes (\otimes^{p}Y)).$$
The conclusion thus follows from Lemma~\ref{LemSubQuo}(i) for the diagonal embedding $\SG_{p}\hookrightarrow \SG_{p}\times \SG_{p}$.\end{proof}

\begin{rem}\label{RemMon} 
\begin{enumerate}[(i)]
\item If we have $p=2$ and $X\in\bT$, we have a short exact sequence
$$0\to\Lambda^2X\to\Gamma^2X\to \Fr_+X\to 0$$
and one can check directly that $\Fr_+$ is a symmetric monoidal functor.
\item By Example~\ref{ExVec}, the commutative diagram~\eqref{eqCommD} and Lemmata~\ref{Lemj} and~\ref{LemFaith}, we find that in tannakian categories we have $\Fr_+(X)\otimes \Fr_+(Y)\simeq\Fr_+(X\otimes Y)$.
\item Example~\ref{ExamTwist}(ii) and Lemma~\ref{Lempowers} similarly show that in tensor categories which admit a fibre functor over an algebra in $\ver_p$, we have $\Fr_+^{(j)}\simeq \Fr_+^j$.
\end{enumerate}
\end{rem}

\begin{lemma}\label{LemSFr}Let $X$ and $S$ be self-dual objects, where $S$ is simple. If $[\otimes^{p^j}X:S]=1=[\Sym^{p^j}X:S]$, then $[\Fr_+^{(j)}X:S]=1$.
\end{lemma}
\begin{proof}
By self-duality we also have $[\Gamma^{p^j}X:S]=1$. Since $[\otimes^{p^j}X:S]=1$, it follows that the unique subquotient isomorphic to $S$ must be a subquotient of the image of $\Gamma^{p^j}X\to \Sym^{p^j}X$.
\end{proof}

\subsection{The skew symmetric and internal twist}
\subsubsection{} For $X\in\bC$ we can restrict the $\SG_p$-action on $\otimes^pX$ to the subgroup $\CG_p<\SG_p$, yielding
$$\bC\to\Rep(\CG_{p},\bC),\quad X\mapsto \otimes^pX.$$

\begin{ddef}
The {\bf internal Frobenius twist} is the functor
$$\Fr_{{\rm in}}=\Triv_{\CG_p}\circ \otimes^p:\;\bC\,\to\, \bC.$$
\end{ddef}

\begin{lemma}\label{AddFrobin}
The functor $\Fr_{{\rm in}}$ is additive. 
Moreover, a tensor category $\bT$ is locally semisimple if and only if $\Fr_{\innn}:\bT\to\bT$ is exact.
\end{lemma}
\begin{proof}
Additivity follows as in Lemma~\ref{AddFrob}, using Corollary~\ref{CorIndex}(ii).
Now consider a tensor category $\bT$. By Lemma~\ref{DirSumm}(iii), the functor $\Fr_{\innn}$ contains $\Fr_+$ as a direct summand. Hence Proposition~\ref{CoolThm} shows that if $\Fr_{\innn}$ is exact, $\bT$ must be locally semisimple. The claim in the other direction follows as in the proof of Proposition~\ref{CoolThm}.
\end{proof}

For the rest of the subsection we assume that $p>2$. 
\begin{ex}\label{ExamTwistin}
Set $\bT=\svec$ and take $V\in \svec$. We have $\Fr_{\innn}(V)=V^{(1)}\simeq V$, the ordinary Frobenius twist of $V$ as a $\mk$-module.
\end{ex}

\begin{ddef}
For $j\in\mN$, the $j$-th {\bf skew symmetric} Frobenius twist is the functor
$$\Fr^{(j)}_{-}=\Triv_{\SG_{p^j}}\circ( \sgn\otimes) \circ \otimes^{p^j}:\;\bC\,\to\, \bC ,\quad X\mapsto \Triv_{\SG_{p^j}}(\sgn\otimes (\otimes^{p^j}X)).$$
\end{ddef}
The following lemma follows from the definition and as above.
\begin{lemma}\label{LemSkewTwist} Take $j\in\mN$.
\begin{enumerate}[(i)]
\item $\Fr_-^{(j)}$ is additive.
\item $\Fr_-^{(j)}$ is exact if $\bC$ is a locally semisimple tensor category.
\item In $\bC\boxtimes\svec$ we have
$$\Fr_{-}^{(j)}(X)\boxtimes \bar{\unit}\;\simeq\; \Fr_+^{(j)}(X\boxtimes \bar{\unit}).$$
\end{enumerate}
\end{lemma}

\begin{qu}Let $\bT$ be a tensor category.
\begin{enumerate}[(i)]
\item If $p=3$, one finds $\Fr_{\innn}=\Fr_+\oplus \Fr_-$. Is the same equation true for $p>3$?
\item If $p=3$, is $\Fr_{\innn}$ monoidal? Closely related, if $p=3$, is $\Fr_{\innn}\simeq\Fr$, with $\Fr$ as in Definition~\ref{DefFr} below? 
\item Do we have $\Fr_+\circ\Fr_-=0=\Fr_-\circ \Fr_+$?
\end{enumerate}
\end{qu}

\subsection{The external twist}
Recall the semisimplification functor $S$ from Definition~\ref{defss} and the Verlinde category $\ver_p=\overline{\Rep \CG_p}$ in~\ref{VerpSub}.
\begin{ddef}\label{DefFr}
The {\bf external Frobenius twist} is the functor
$$\Fr= S_{\CG_p}\circ \otimes^p :\,\bC\to\bC\boxtimes \ver_p,\quad X\,\mapsto\, \bigoplus_{i=1}^{p-1}\Triv_{\CG_p}(M_i\otimes (\otimes^p X))\boxtimes\overline{M}_i.$$
\end{ddef}
Note that we have $\Fr_+\simeq\Fr$ if $\charr(\mk)=2$.

\begin{lemma}
The functor $\Fr$ is additive. If $\bT$ is a tensor category, then $\Fr$ is exact if and only if $\bT$ is locally semisimple.
\end{lemma}
\begin{proof}
That $\Fr$ is additive follows as in the proof of Lemma~\ref{AddFrob}, using additionally Lemma~\ref{TensorInd}. The statement about locally semisimple tensor categories follows as in Lemma~\ref{AddFrobin}.\end{proof}

\begin{prop}
If $\bT$ is semisimple and schurian, then $\Fr$ coincides with the similarly denoted functor in \cite[Definition~3.5]{Ostrik}.
\end{prop}
\begin{proof}
This follows by comparing the definitions and applying Proposition~\ref{propsimplification}.
\end{proof}


\section{Characterising locally free objects}\label{SecTanObj}
We fix a field $\mk$ with $p:=\charr(\mk)>0$ and a tensor category $\bT$ over $\mk$. We will provide internal characterisations for locally $\bV$-free objects for $\bV$ equal to $\vecc$ and
$\svec$.

\subsection{Locally free objects}
By {\bf locally free objects} in $\bT$ we refer to locally $\vecc$-free objects. We will also replace $\bT\boxtimes\vecc$ by the equivalent $\bT$ in all definitions. Since $\scA\otimes X\simeq \scA^{\oplus \alpha}$ 
for $X\in\bT$ and some cardinality $\alpha$ implies that $\alpha$ is finite, by Lemma~\ref{LemFaith},
there is no need to specify to `locally free objects of finite rank'.

\begin{ddef}\label{DefTanObj}
For $X\in\bC$, with $\bC$ a monoidal category as in \ref{SMC}, we define 
$$[X]_{\unit} \;=\; \sup\{n\in\mN\,|\, \Lambda^n( \Fr_+^{(j)}X)\not=0,\,\mbox{ for all}\; j\in\mN\}\;\; \in\,\mN\cup\{\infty\}.$$
\end{ddef}


\begin{thm}\label{ThmTanObj}
An object $X\in\bT$  is locally free if and only if one of the following equivalent conditions is satisfied:
\begin{enumerate}[(i)]
\item $[X]_{\unit}\in\mN$ and $\Lambda^rX=0$ if $r>[X]_{\unit}$.
\item \begin{enumerate}
\item there exists $n\in \mN$ such that $\Lambda^nX=0$, and
\item if $\Lambda^n\Fr^{(j)}_+(X)=0$ for some $j,n\in\mN$, then also $\Lambda^nX=0$.
\end{enumerate}
\item There exists a symmetric monoidal category $\bC$ as in \ref{SMC}, a tensor functor $F:\bT\to\bC$ and $m\in\mN$ such that $F(X)\simeq\unit^{\oplus m}$ in $\bC$.
\end{enumerate}
\end{thm}

\begin{ex}\label{ExLam2}
If $\Lambda^2X=0$ and $X\not=0$, then $X$ is locally free of rank 1. Indeed, by Theorem~\ref{ThmTanObj}(ii) it suffices to check that $\Fr^{(j)}_+X\not=0$ for all $j\in\mN$. This follows from the assumption $\Lambda^2X=0$ which implies that $\Sym^nX=\otimes^nX=\Gamma^nX$ for all $n$, so in particular $\Fr^{(j)}_+X\simeq \otimes^{p^j}X$.
\end{ex}

We will show in Lemma~\ref{LemBE} below that, at least for $p=2$, characterisation \ref{ThmTanObj}(ii) is strict in the sense that one cannot restrict to finite a number of $j\in\mN$.
Theorem~\ref{ThmTanObj}(iii) is also proved, in slightly different generality, in \cite[Proposition~4.3.8]{Sch2}.
Before proving the theorem, we derive some properties for a monoidal category~$\bC$ as in \ref{SMC}.
\begin{lemma}\label{LemExtra}
For $Y,Z\in\bC$, we have
\begin{enumerate}[(i)]
\item $ [\unit\oplus Y]_{\unit}=1+[Y]_{\unit}$;
\item $[Y\oplus Z]_{\unit}=0\;$ if $\;[Y]_{\unit}=0=[Z]_{\unit}$;
\item $[Y]_{\unit}=0$ if and only if there exists $k\in\mN$ such that $\Fr_+^{(j)}Y=0$ for all $j\ge k$.
\end{enumerate}
\end{lemma}
\begin{proof} 
Part (i) follows from Lemma~\ref{AddFrob} and equation~\eqref{PropLam}.
 Part (iii) follows from Lemma~\ref{Lemnp1}. Part (ii) follows from part (iii) and Lemma~\ref{AddFrob}.
\end{proof}

The following result is in the proof of \cite[Lemme~2.8]{Del02}. This is precisely the part of the proof which does not rely on the assumption of characteristic zero.
\begin{lemma}[Deligne] \label{LemDeldelta}
Assume that $\bC$ admits countable coproducts which are preserved by $-\otimes-$ and let $M\in\bC$ be dualisable. For any $(\scA,m_{\scA},\eta_{\scA})\in\alg{\bC}$, we have that $\scA$ is a direct summand of $\scA\otimes M$ in $\bC_{\scA}$ if and only if there exists an algebra morphism
$$\Sym^\bullet(M)\otimes\Sym^\bullet(M^\vee)\;\stackrel{f}{\to}\; \scA\quad\mbox{with}\quad f\circ \co_M=\eta_{\scA}.$$
\end{lemma}

\begin{lemma}\label{LemTrivABC}
Consider a dualisable object $V$ in $\bC$ with quotient $V\stackrel{\pi}{\tto}W$
 and dualisable subobject $U\stackrel{\iota}{\hookrightarrow}V$. The composition $\pi\circ\iota$ is zero if and only if $(\pi\otimes \iota^t)\circ\co_V$ is zero.
 \end{lemma}
 \begin{proof}
 The canonical isomorphism
 $$\Hom(U,W)\;\stackrel{\sim}{\to}\; \Hom(\unit,  W\otimes U^\vee),\quad g\mapsto (g\otimes {U^\vee})\circ \co_U,$$
 maps $\pi\circ\iota$ to $(\pi\otimes \iota^t)\circ\co_V$.
 \end{proof}

\begin{cor}\label{CorMj}
Assume that $\bC$ admits countable coproducts which are preserved by $-\otimes-$ and let $M\in\bC$ be dualisable with $\Sym^n( M^\vee)$ flat (and hence dualisable) for all $n\in\mN$. There exists a non-zero $\scA\in\alg{\bC}$ for which $\scA$ is a direct summand of $\scA\otimes M$ in $\bC_{\scA}$ if and only if $[M]_{\unit}>0$.
\end{cor}
\begin{proof}
We start from Lemma~\ref{LemDeldelta}. Like all algebra morphisms, any $f$ as in Lemma~\ref{LemDeldelta} is assumed to satisfy $f\circ\eta=\eta_{\scA}$ with $\eta$ the unit of the algebra $\Sym^\bullet(M)\otimes\Sym^\bullet(M^\vee)$. The existence of $\scA$ is thus equivalent to the quotient of $\Sym^\bullet(M)\otimes\Sym^\bullet(M^\vee)$ with respect to the ideal generated by $(\eta-\co_M)(\unit)$ being non-zero. As argued in the proof of \cite[Lemme~2.8]{Del02} this is equivalent to the composition
$$\unit\;\stackrel{\otimes^n\co_M}{\hookrightarrow}\; (\otimes^nM)\otimes (\otimes^nM^\vee)\;\tto\; \Sym^n(M)\otimes \Sym^n(M^\vee) $$
being non-zero for all $n\in\mN$.

By Lemma \ref{LemTrivABC} this is equivalent to the composition
$${\mathrm{H}}^0(\SG_n,\otimes^nM)=\Gamma^n(M)\hookrightarrow \otimes^nM\tto \Sym^nM={\mathrm{H}}_0({\SG_n},\otimes^n M)$$
being non-zero. The latter just means that $\Triv_{\SG_n}(\otimes^nM)$ is never zero. By Lemma~\ref{Lemnp1} the condition is thus equivalent to $ \Fr_+^{(j)}M\not=0$ for all $j\in\mN$. 
\end{proof}

\begin{prop}\label{Thm2}
For $X\in\bT$ and $d\in\mN$,
we have $[X]_{\unit}\ge d$ if and only if there exists a non-zero $\scA\in \Alg{\bT}$ and $N\in\Mod_{\scA}$ such that
$$\scA\otimes X\;\simeq\; \scA^{\oplus d}\,\oplus \, N,\quad\mbox{in $\Mod_{\scA}$.}$$
\end{prop}
\begin{proof}
We start by observing that for any non-zero $\scR\in\Alg\bT$, we have
\begin{equation}\label{brackR}[\scR\otimes X]_{\unit}\;=\;[X]_{\unit},\quad\mbox{for $X\in\bT$,}\end{equation}
since $\scR\otimes-$ is a (faithful exact) tensor functor.
If $\scA^{\oplus d}$ is a direct summand of $\scA\otimes X$, then Lemma~\ref{LemExtra}(i) and equation~\eqref{brackR} imply that $[X]_{\unit}\ge d$.

To prove the other direction, we apply induction on $d$. If $d=0$ there is nothing to prove. Assume that the claim is true for $d-1$. Hence, if $[X]_{\unit}\ge d$ we know that there exists $\scB$ in $\Alg{\bT}$ and $M$ in $\Mod_{\scB}$ such that 
$$\scB\otimes X\;\simeq\;\scB^{\oplus (d-1)}\oplus M.$$
By Lemma~\ref{LemExtra}(i) and equation~\eqref{brackR} we have $[M]_{\unit}>0$. By construction, $\Sym^n_{\scB}(M^\vee)$ is a direct summand of $\scB\otimes\Sym^n(X^\vee)$ and therefore dualisable, so in particular flat.
 By Corollary~\ref{CorMj} for $\bC=\Mod_{\scB}$, there exists $\scA$ in $\alg{\Mod_{\scB}}$, which we can also interpret in $\Alg{\bT}$, for which
$$\scA\otimes X\;\simeq\; \scA\otimes_{\scB}(\scB\otimes X)\;\simeq\; \scA^{\oplus d-1}\oplus \scA\otimes_{\scB}M\;\simeq\;\scA^{\oplus d}\oplus N,$$
which concludes the proof.
\end{proof}

\begin{proof}[Proof of Theorem~\ref{ThmTanObj}]
Take $X$ as in (i) and set $d:=[X]_{\unit}$. By Proposition~\ref{Thm2}, there exists $\scA\in\Alg\bT$ such that
$\scA\otimes X$ is of the form $\scA^{\oplus d}\oplus N$. By assumption and \eqref{PropLam}, we have
$$0\;\simeq\; \scA\otimes \Lambda^{d+1}X\;\simeq\;\Lambda^{d+1}_{\scA}(\scA\otimes X)\;\simeq\;\bigoplus_{i=0}^{d}\left(\Lambda_{\scA}^{i+1}N\right)^{\oplus \binom{d}{i}},$$
which implies $N=0$. Hence (i) implies that $X$ is locally free. If $X$ is locally free, it clearly satisfies (iii).
That (iii) implies (ii) follows from the fact that $F$ is a (faithful exact) tensor functor. That (ii) implies (i) is straightforward.
\end{proof}

\begin{lemma}\label{LemExEx}
For $X,Y\in\bT$ with $[X]_{\unit},[Y]_{\unit}\in\mN$ we have $[X\oplus Y]_{\unit}=[X]_{\unit}+[Y]_{\unit}$.
\end{lemma}
\begin{proof}
By Proposition~\ref{Thm2}, there exists a non-zero $\scA\in\Alg\bT$ such that
$$\scA\otimes X\;\simeq\;  \scA^{\oplus [X]_{\unit}}\oplus M\quad\mbox{and}\quad
\scA\otimes Y\;\simeq\; \scA^{\oplus [Y]_{\unit}}\oplus N,$$
for $M,N$ in $\Mod_{\scA}$. By Lemma~\ref{LemExtra}(i) and equation~\eqref{brackR} we find that $[M]_{\unit}=0=[N]_{\unit}$. By Lemma~\ref{LemExtra}(i) and (ii) we then find $[\scA\otimes (X\oplus Y)]_{\unit}=[X]_{\unit}+[Y]_{\unit}$ and the conclusion follows from equation~\eqref{brackR}.
\end{proof}

\begin{lemma}\label{LemBE}
Assume $\mk$ is algebraically closed with $\charr(\mk)=2$. For each $n\in\mZ_{>0}$ there exists a tensor category $\bT_n$ with $X_n\in \bT_n$ such that $\Lambda^3X_n=0$ and
$$\Lambda^d \Fr^{(j)}_+X_n=0\quad\Rightarrow\quad \Lambda^d X_n=0,\qquad\mbox{for all $d\in\mN$ and $j< n$,}$$
but $X_n$ is not locally free.
\end{lemma}
\begin{proof}
For $n\in\mN$, let $\bT_n$ be the tensor category $\mathcal{C}_{2n}$ of \cite[Theorem~2.1]{BE},  which contains $\cC_{2n-2}$ as subcategory if $n>0$. For $X_n$ we take the similarly named self-dual simple object in \cite[2.1(iii)]{BE}, so $X_n=0$ if and only if $n=0$. It follows from \cite[2.1(vi) and (ix)]{BE} that
\begin{equation}\label{eqBE}[\otimes^{2^j}X_n:X_{n-j}]=1\,\mbox{ and }\;[\otimes^{2^i}X_n:X_{n-j}]=0,\quad\mbox{for $i<j<n$}.\end{equation}
By \cite[2.1(iii) and (vii)]{BE} we have $\Lambda^2X_n=\unit$ and $\Lambda^3X_n=0$, if $n>0$. The former of the two equations implies $[\Sym^{2^j}X_n:X_{n-j}]=1$, by equations~\eqref{eqSym} and~\eqref{eqBE}. By Lemma~\ref{LemSFr}, we thus have $[\Fr^{(j)} X_n:X_{n-j}]=1$ for $j<n$. It follows easily from \cite[2.1(xi)]{BE} that $\Fr^{j}X_n=X_{n-j}$ for $j\le n$. By Lemma~\ref{Lempowers}, we thus have $\Fr^{(j)}X_n=X_{n-j}$ for $j\le n$.
\end{proof}

\subsection{Locally super free objects} In this subsection we assume that $p\not=2$.

\begin{ddef}
For $X\in\bC$ with $\bC$ as in \ref{SMC} we define $[X]_{\bar{\unit}}\in\mN\cup\{\infty\}$ as 
$$[X]_{\bar{\unit}} \;=\; \sup\{n\in\mN\,|\, \Gamma^n( \Fr_-^{(j)}X)\not=0,\,\mbox{ for all}\; j\in\mN\}.$$
\end{ddef}


\begin{prop}\label{Thm2s}
For $X\in\bT$ and $d,d'\in\mN$,
we have $[X]_{\unit}\ge d$ and $[X]_{\bar\unit}\ge d'$ if and only if there exists non-zero $\scA\in \Alg(\bT\boxtimes \svec)$ and $N\in\Mod_{\scA}$ such that
$$\scA\otimes X\;\simeq\; \scA\otimes(\unit^{\oplus d}\oplus\bar{\unit}^{\oplus d'})\,\oplus \, N,\quad\mbox{in $\Mod_{\scA}$.}$$
\end{prop}
\begin{proof}
This is proved similarly to Proposition~\ref{Thm2}, using Lemma~\ref{LemSkewTwist}(iii).
\end{proof}

Following \cite[\S 3.3]{SchappiNew}, for $a,b\in\mZ_{>0}$, we consider the element $\Sigma^{a,b}$ in the group ring $\mZ S_{ab}$ of the symmetric group which is the (unnormalised) Young symmetriser corresponding to the partition $(a^b)$. For an object $X$ in a symmetric monoidal category $\bC$ as in \ref{SMC}, we have the corresponding endomorphism $\Sigma_X^{a,b}$ of $\otimes^{ab} X$. In particular, $\Lambda^n X$ is the image of $\Sigma^{1,n}_X$.

\begin{thm}\label{ThmSTO}
An object $X\in\bT$ is locally $\svec$-free if and only if $([X]_{\unit},[X]_{\bar{\unit}})\in\mN\times\mN$ and $\Sigma_X^{[X]_{\bar{\unit}}+1,[X]_{\unit}+1}=0$.
\end{thm}

\begin{proof}
One direction of the claim follows from \cite[Lemma~4.2.5]{SchappiNew}. Now assume that $(m,n):=([X]_{\unit},[X]_{\bar{\unit}})$ is as in the theorem.
 By Proposition~\ref{Thm2s}, there exists non-zero $\scA\in\Alg(\bT\boxtimes\svec)$ for which 
 $$\scA\otimes X\;\simeq\; \scA\otimes(\unit^{\oplus m}\oplus\bar{\unit}^{\oplus n})\,\oplus \, N.$$
By \cite[Proposition~3.3.6]{SchappiNew}, the fact that $\Sigma_X^{[X]_{\bar{\unit}}+1,[X]_{\unit}+1}=0$, and hence $\Sigma_{\cA\otimes X}^{[X]_{\bar{\unit}}+1,[X]_{\unit}+1}=0$ implies that $N=0$.
 \end{proof}
 
 \begin{ex}\label{ExS2}
 If $\Sym^2X=0$ then $X$ is locally $\svec$-free of rank $0|1$.
 \end{ex}


\section{Internal characterisations}\label{SecTanCat}

\subsection{Tannakian categories} 
Fix a field $\mk$ of characteristic $p=\charr(\mk)>0$.
The following generalises \cite[Th\'eor\`eme~7.1]{Del90} to fields of positive characteristic.
\begin{thm}\label{Thm3}
For a tensor category $\bT$ the following conditions are equivalent:
\begin{enumerate}[(i)]
\item $\bT$ is tannakian.
\item For every $X$ in $\bT$,
\begin{enumerate}
\item there exists $n\in \mN$ such that $\Lambda^nX=0$;
\item if $\Lambda^n\Fr^{(j)}_+(X)=0$ for some $j,n\in\mN$, then also $\Lambda^nX=0$.
\end{enumerate}
\item Every $X$ in $\bT$ is locally free.
\end{enumerate}
\end{thm}
\begin{proof}
First we prove that (i) implies (ii). If $\bT$ is Tannakian, it admits a tensor functor to $\vecc_{K}$ for some field extension $K/\mk$, by Lemma~\ref{LemFinLen}. The properties in (ii) are satisfied in $\vecc_K$, since the objects $\Lambda^nX$ and $\Fr_+^{(j)}X$ are the same for $\vecc_K$ considered as a $K$-linear or $\mk$-linear category. By Lemma~\ref{LemFaith} and diagram~\eqref{eqCommD}, they are thus satisfied in $\bT$ as well. 
 Theorem~\ref{ThmTanObj} states that (ii) implies (iii).
 Proposition~\ref{PropApp}(iii) shows that (iii) implies (i).
\end{proof}

\begin{prop}\label{TannSub} 
Any tensor category $\bT$ has a unique maximal tannakian subcategory, the tensor subcategory of locally free objects. If $\bT$ is locally semisimple, the latter is a Serre subcategory.
\end{prop}
\begin{proof}
By Theorem~\ref{Thm3} it suffices to prove that the full subcategory of locally free objects is a tensor subcategory, respectively a Serre subcategory.
These are special cases of Lemma~\ref{FreeTensorN}. \end{proof}

\begin{rem}
\begin{enumerate}[(i)]
\item  By Lemma~\ref{FreeTensorN}(ii), we can simplify Theorem~\ref{Thm3} as follows. A tensor category $\bT$ is tannakian if and only if $\bT=\langle E\rangle$ for a set $E$ of locally free objects in $\bT$. In particular, if $\bT$ is finitely generated, $\bT=\langle Y\rangle$ for $Y\in\bT$, it suffices to check condition~\ref{Thm3}(ii) on $Y$. 
\item In subsequent work in \cite[Proposition~7.2]{CE}, it is shown that if $p>2$ and $\mk$ is algebraically closed, the condition that $\bT$ be locally semisimple is redundant for the maximal tannakian subcategory to be a Serre subcategory. Note that for $p=2$, the condition is necessary, see  \cite[Example~3.3]{EHO} or \cite{BE}.
\end{enumerate}
\end{rem}

\subsection{Super tannakian categories}
Fix a field $\mk$ of characteristic $p=\charr(\mk)>2$.
\begin{thm}\label{ThmSuper}
For a tensor category $\bT$, the following conditions are equivalent:
\begin{enumerate}[(i)]
\item $\bT$ is super tannakian.
\item For every $X\in \bT$, we have that $(m,n):=([X]_{\unit},[X]_{\bar{\unit}})\in\mN\times\mN$ and $\Sigma_X^{n+1,m+1}=0$.
\item Every $X$ in $\bT$ is locally $\svec$-free.
\end{enumerate}
\end{thm}

\begin{proof}[Proof of Theorem~\ref{ThmSuper}]
That (i) implies (ii) is proved as in the proof of Theorem~\ref{Thm3}. That (ii) implies (iii) is in Theorem~\ref{ThmSTO}. That (iii) implies (i) is in Proposition~\ref{PropApp}.
\end{proof}

\begin{prop}
Any tensor category $\bT$ has a unique maximal super tannakian subcategory. If $\bT$ is locally semisimple, the latter is a Serre subcategory.
\end{prop}
\begin{proof}
Mutatis mutandis Proposition~\ref{TannSub}.\end{proof}

\begin{ex}\label{ExPointed}
Let $\bV$ be a semisimple pointed tensor category. In particular, for simple $S\in\bV$, the object $\otimes^nS$ is simple, for all $n\in\mN$. We thus either have $\Sym^2S=0$ or $\Lambda^2S=0$. By Examples~\ref{ExLam2} and \ref{ExS2}, $S$ is locally $\svec$-free. It follows from Theorem~\ref{ThmSuper} that $\bV$ is super tannakian. If $p=2$ one shows similarly that any pointed semisimple tensor category is tannakian, since then $\Gamma^2X=0$ implies $\Lambda^2X=0$, see Remark~\ref{RemMon}(i), and thus $X=0$.
\end{ex}

\subsection{Affine group schemes} In order to proceed to the last section on neutrality of tannakian categories we need a short interlude on affine group schemes.
We refer to \cite{Wa} for the basic notions and to \cite{Mas} for the corresponding results for supergroup schemes. We prove some facts which are presumably well-known, but for which we did not find references.

Consider an affine supergroup scheme $G$ with homomorphism $p:\mZ/2\to G$ inducing the grading on $\mk[G]$, see \cite[\S 8.19]{Del90}. We denote by $\Rep(G,p)$  the category of $G$-representations in $\svec$ which yield the canonical $\mZ/2$-action on super spaces via $p$. As a special case, we can consider an affine group scheme as an affine supergroup scheme and set $p$ to be the trivial homomorphism.  Then $\Rep(G,p)$ corresponds to the ordinary category of $G$-representations in $\vecc$. For a closed sub(super)group $H<G$ we denote by $\Rep(G,p)^H$ the tensor subcategory of all representations for which $H$ is in the kernel. \begin{thm}\label{ThmSubCat}
Consider an affine supergroup scheme $G$ with $p:\mZ/2\to G$ as above. There is a bijection between closed normal subgroups $N\lhd G$ and tensor subcategories of $\Rep(G,p)$:
$$N\;\mapsto\; \Rep(G,p)^N.$$
Furthermore, the essential image of the canonical tensor functor $\Rep(G/N,p)\to\Rep(G,p)$ is precisely $\Rep(G,p)^N$.
\end{thm}
\begin{proof}
The statement about the essential image follows immediately from the existence of the quotient $G/N$ in \cite[Theorem~16.3]{Wa} and its universality, see \cite[Theorem~15.4]{Wa}.
 
For any tensor subcategory $\bT\subset \Rep(G,p)$, composition of the inclusion functor and the forgetful functor $\Rep(G,p)\to\svec$ yields a fibre functor $\bT\to\svec$. By Lemma~\ref{LemTannForm}(i), there exists a super group scheme $H$ under $\mZ/2$ with a tensor equivalence $\Rep(H,p)\simeq\bT$. Since the functor $\Rep(H,p)\simeq\bT\hookrightarrow \Rep(G,p)$ admits a commutative diagram (up to isomorphism) with the forgetful functors to $\svec$, it follows from Lemma~\ref{LemTannForm}(i) that it induces a homomorphism $G\to H$ under $\mZ/2$, which in turns induces the functor $\Rep(H,p)\to \Rep(G,p)$, up to isomorphism.
By \cite[Proposition~2.21(a)]{DM}, the morphism $\mk[H]\to\mk[G]$ is injective, so by definition in \cite[\S 15.1]{Wa}, $H$ is a quotient of $G$ and hence of the form $G/N$ for a closed normal subgroup $N$, see \cite[Corollary~16.3]{Wa}. 

By the first paragraph, the above assignment of a normal subgroup to a tensor subcategory is a two-sided inverse of the function in the theorem.
\end{proof}

\begin{cor}\label{CorLatt}
Suppose that under the bijection in Theorem~\ref{ThmSubCat}, we have $N_1\mapsto \bT_1$ and $N_2\mapsto\bT_2$, for closed normal subgroups $N_1,N _2\lhd G$. Then we also have $N_1\cap N_2\mapsto \langle \bT_1,\bT_2\rangle$ and $N_1N_2\mapsto \bT_1\cap\bT_2$.
\end{cor}
\begin{proof}
By construction, the bijection is order reversing, for the inclusion orders on both sets. The partially ordered sets are actually lattices, so the join and meet will be interchanged.
\end{proof}

\begin{cor}\label{CorFG}
For $p:\mZ/2\to G$ as in Theorem~\ref{ThmSubCat} and $\bT=\Rep(G,p)$, the following are equivalent:
\begin{enumerate}[(i)]
\item $\bT$ is finitely generated as a tensor category;
\item every tensor subcategory of $\bT$ is finitely generated;
\item $G$ is algebraic, {\it i.e.} $\mk[G]$ is finitely generated as an algebra.
\end{enumerate}

\end{cor}
\begin{proof}

For brevity we leave out reference to `super'.
Condition (iii) implies that the topological space underlying the scheme $G=\Spec \mk[G]$ is noetherian. As the poset of closed normal subgroups is a sub-poset of the poset of all closed subspaces, the implication $(iii)\Rightarrow (ii)$ follows from Theorem~\ref{ThmSubCat}. Clearly (ii) implies (i). That (i) implies (iii) is in \cite[Proposition~2.20(ii)]{DM}.
\end{proof}


\begin{prop}\label{PropSheaf}
Let $G$ be an affine supergroup scheme and $N_1,N_2$ closed normal subgroups with $N_1\cap N_2=1$. The canonical homomorphism
$$G\;\to\;G/N_1\times_{G/N_1N_2}G/N_2$$
is an isomorphism. Equivalently, we have $\mk[G/N_1]\otimes_{\mk[G/N_1N_2]}\mk[G/N_2]\stackrel{\sim}{\to}\mk[G]$.
\end{prop}
\begin{proof}For brevity we write the proof only for groups.
For a closed normal subgroup $N\lhd G$, denote by $D^G_N$ the quotient of $G$ and $N$ as functors $\mathbf{Alg}_{\mk}\to\Grp$. By definition of $G/N$, see \cite[\S 16.3]{Wa}, $D^G_N$ is a subfunctor of $G/N$. We clearly have an isomorphism of group functors
$$G\;\stackrel{\sim}{\Rightarrow}\;D^G_{N_1}\times_{D^G_{N_1N_2}}D^G_{N_2}.$$
As spelled out in \cite[Theorem~15.5]{Wa}, $D^G_N$ is a `fat subfunctor' of $G/N$. It follows as an easy exercise that the right-hand side of the above equation is therefore a fat subfunctor of $G/N_1\times_{G/N_1N_2}G/N_2$ and consequently that the homomorphism in the proposition is the inclusion of a fat subfunctor. However, since $G$ is itself a sheaf for the fpqc topolgy on $(\mathbf{Alg}_{\mk})^{\op}$, see \cite[\S 15.6]{Wa}, such an inclusion of $G$ must be an isomorphism.
\end{proof}

\begin{lemma}\label{LemGNk} If $\mk$ is algebraically closed, $G$ is an algebraic supergroup over $\mk$ and $N$ a closed normal subgroup, then the canonical left exact sequence
$$1\to N(\mk)\to G(\mk)\to G/N(\mk)\to 1$$
is exact.
\end{lemma}
\begin{proof}For groups, this is
\cite[Theorem~15.2]{Wa}. For supergroups, the result is a consequence of the former and \cite[Theorem~3.13(3)]{Mas}.
\end{proof}

 \subsection{Neutrality over algebraically closed fields} Fix an {\em algebraically closed} field $\mk$ of arbitrary characteristic.
Deligne announced in \cite{Del02} that over $\mk$ all tannakian categories are neutral. However, the proof was deemed `too painful' to add. In his letter \cite{Del11}, Deligne sketched the argument, and it was written out in more detail by the author in [Appendix A, arXiv:1812.02452v2]. In this section we present a variation of this argument based on the notion of splitting algebras. 

\begin{thm}\label{NewThmNeut}
If $\mk=\overline{\mk}$ then any (super) tannakian category is neutral. Moreover, any two tensor functors to $\vecc$ or $\svec$ from a fixed tensor category are isomorphic.
\end{thm}

For finitely generated tensor categories, Theorem~\ref{NewThmNeut} is proved in \cite[Corollaire~6.20]{Del90} and \cite[Proposition~4.5]{Del02}. 

\begin{lemma}\label{SplitFG}
Let $\bT$ be a finitely generated (super) tannakian tensor category.
\begin{enumerate}[(i)]
\item $\bT$ admits a neutral ($\svec$-)splitting algebra, unique up to isomorphism.
\item For a neutral ($\svec$-)splitting algebra $\scA$ for $\bT$ and a tensor subcategory $\bT'\subset \bT$ any algebra endomorphism of $\Gamma_{\bT'}\scA$ lifts to one of $\scA$.
\end{enumerate}
\end{lemma}
\begin{proof}
By Theorem~\ref{ThmFibreNeut}, part (i) is a reformulation of the claim that every finitely generated (super) tannakian category is neutral (as we observed above) and that fibre functors to $\vecc$ or $\svec$ are unique up to isomorphism, see Proposition~\ref{PropUnique}.

By Theorem~\ref{ThmFibreNeut}, endomorphisms of a neutral splitting algebra correspond to the automorphisms of the associated fibre functor. By Lemma~\ref{LemTannForm}, the are the rational $\mk$-points of the corresponding affine (super)group scheme. For $\bT$, this group scheme, say $G$, is algebraic by Corollary~\ref{CorFG}. By Theorem~\ref{ThmSubCat}, the subcategory $\bT'$ corresponds to a normal subgroup $N\lhd G$ and algebra endomorphisms of $\Gamma_{\bT'}\scA$ are to be identified with $\mk$-points of $G/N$. Part (ii) is therefore a reformulation of Lemma~\ref{LemGNk}.
\end{proof}

\begin{lemma}\label{LemSlick}
Consider a tensor category $\bT$ with tensor subcategories $\bT_1$ and $\bT_2$ such that $\langle\bT_1,\bT_2\rangle=\bT$ and $\bT_2$ is finitely generated. 
\begin{enumerate}[(i)]
\item If $\scA_1$ and $\scA_2$ are neutral ($\svec$-)splitting algebras of respectively $\bT_1$ and $\bT_2$, then there exists an algebra morphism $\scA_{12}:=\Gamma_{\bT_2}\scA_1\to \scA_2$. For any such morphism the associated $\scA:=\scA_1\otimes_{\scA_{12}}\scA_2$ is a neutral ($\svec$-)splitting algebra for $\bT$.
\item For a neutral ($\svec$-)splitting algebra $\scA$ for $\bT$, any endomorphism of $\Gamma_{\bT_1}\scA$ lifts to $\scA$.
\end{enumerate}
\end{lemma}
\begin{proof} We leave out `super' from the proof for brevity. We will freely use that morphisms between splitting algebras (for a fixed tensor category) are always isomorphisms, by Theorem~\ref{ThmFibreNeut}. 

We prove part (i). Since $\bT_1$ and $\bT_2$ are (neutral) tannakian, $\bT=\langle \bT_1,\bT_2\rangle$ is also tannakian, by Lemma~\ref{FreeTensorN}(ii) and Theorem~\ref{Thm3}. We also observe that $\bT_1\cap\bT_2$ is finitely generated by Corollary~\ref{CorFG}. 

First we consider the special case where $\bT_1$ is finitely generated as well. Then the tannakian category $\bT$ is also finitely generated, so of the form $\Rep G$, for an algebraic group $G/\mk$. By Theorem~\ref{ThmSubCat}, we can associate normal subgroups $N_i\lhd G$ to $\bT_i\subset \bT$. By Corollary~\ref{CorLatt}, we have $N_1\cap N_2=1$. By the uniqueness of neutral splitting algebras for $\bT_i$ in Lemma~\ref{SplitFG}(i) and Example~\ref{examkGsplits} we find the left and right vertical isomorphisms in the diagram
$$\xymatrix{
\scA_1\ar[d]^{\sim}&&\ar@{_{(}->}[ll]\scA_{12}\ar[d]^{\sim}&&\scA_2\ar[d]^{\sim}\\
\mk[G/N_1]&&\ar@{_{(}->}[ll]\mk[G/N_1N_2]\ar@{^{(}->}[rr]&&\mk[G/N_2].
}$$
The horizontal morphisms are the inclusion of the subalgebra $\scA_{12}\subset\scA_1$ and the Hopf algebra morphisms induced from $G/N_i\to G/(N_1N_2)$. By \cite[\S 16.3]{Wa}, $\mk[G/N_1N_2]$ is the maximal submodule of $\mk[G/N_1]$ on which $N_1N_2/N_1$ acts trivially. Theorem~\ref{ThmSubCat} thus yields the middle vertical isomorphism which makes the left square commutative. We then choose the morphism $\scA_{12}\to\scA_2$ which creates another commutative square in the above diagram. By Proposition~\ref{PropSheaf}, the algebra $\scA_1\otimes_{\scA_{12}}\scA_2$ is then isomorphic to $\mk[G]$ and hence indeed a splitting algebra.
 Now assume we take a different algebra morphism $\scA_{12}\to\scA_2$. It must be a composite $\scA_{12}\to \Gamma_{\bT_1}\scA_2\hookrightarrow \scA_2$, where the first arrow must be an isomorphism. This means that our new $\scA_{12}\to\scA_2$ is equal to our first choice, up to precomposition with an automorphism of $\scA_{12}$. This automorphism lifts to $\scA_1$ by Lemma~\ref{SplitFG}(ii). This lift yields a canonical isomorphism between the two algebras of the form $\scA_1\otimes_{\scA_{12}}\scA_2$, hence the second is also a neutral splitting algebra.

Now we consider the general case of part (i), meaning that $\bT_1$ need not be finitely generated. Take the set $\{\bT_1^\alpha\}$ of all finitely generated tensor subcategories of $\bT_1$ which contain $\bT_1\cap\bT_2$, and set $\scA^\alpha_1:=\Gamma_{\bT_1^\alpha}\scA_1$. By definition, we have $\scA_{12}\subset\scA_1^\alpha$. Since $\bT_1=\cup_\alpha\bT_1^\alpha$, we have
$$\scA_1\;\simeq\;\varinjlim_\alpha \scA_1^\alpha.$$
By  Corollary~\ref{CorResA}(i), the algebra $\scA_1^\alpha$, respectively $\scA_{12}$ and $\Gamma_{\bT_1}\scA_2$, are neutral splitting algebras for  $\bT_1^\alpha$ respectively $\bT_1\cap\bT_2$. Since $\bT_1\cap\bT_2$ is finitely generated, Lemma~\ref{SplitFG}(i) allows to choose an algebra morphism $\scA_{12}\stackrel{\sim}{\to}\Gamma_{\bT_1}\scA_2 \hookrightarrow \scA_2$.
By the previous paragraph and the fact $\bT^\alpha_1\cap\bT_2=\bT_1\cap\bT_2$, each $\scA^\alpha_1\otimes_{\scA_{12}}\scA_2$ is a neutral splitting algebra for $\langle \bT^\alpha_1,\bT_2\rangle$ and hence
$$\scA\;=\;\scA_1\otimes_{\scA_{12}}\scA_2\;\simeq\;\varinjlim_\alpha (\scA^\alpha_1\otimes_{\scA_{12}}\scA_2)$$
is a splitting algebra for $\bT=\cup_\alpha \langle \bT^\alpha_1,\bT_2\rangle$. By definition of $\Ind\bT$, $\unit$ is compact, so
$$\Hom(\unit, \scA)\simeq \varinjlim_\alpha \Hom(\unit,\scA^\alpha_1\otimes_{\scA_{12}}\scA_2)\simeq\mk,$$
and $\scA$ is neutral.

For part (ii), we consider the canonical morphism 
\begin{equation}\label{eqGGG}\Gamma_{\bT_1}\scA\,\otimes_{\Gamma_{\bT_1\cap\bT_2}\scA}\,\Gamma_{\bT_2}\scA\;\to\;\scA.\end{equation}
By part (i), the left-hand side is again a neutral splitting algebra, so the morphism is an isomorphism. Any automorphism of $\Gamma_{\bT_1}\scA$ restricts to an automorphism of $\Gamma_{\bT_1\cap\bT_2}\scA$ and the latter lifts to an automorphism of $\Gamma_{\bT_2}\scA$ by Lemma~\ref{SplitFG}(ii). Hence any automorphism of $\Gamma_{\bT_1}\scA$ lifts to the left-hand side of \eqref{eqGGG}. Since the isomorphism~\eqref{eqGGG} is canonically under $\Gamma_{\bT_1}\scA$, claim (ii) follows.
\end{proof}


\begin{proof}[Proof of Theorem~\ref{NewThmNeut}]
For brevity we leave out reference to `super' in the proof. Let $\bT$ be a tannakian category. We consider the {\em set} $\Ob\bT/\hspace{-1.5mm}\sim$ of isomorphism classes of objects in $\bT$. By the well-ordering theorem we can choose a well-order $\preceq$ on $\Ob\bT/\hspace{-1.5mm}\sim$ and we denote by $\mathbb{O}=(\Ob\bT/\hspace{-1.2mm}\sim,\preceq)$ the corresponding well-ordered set which we interpret as a category in the canonical way. For each $X\in\mO$ we choose a (unique up to isomorphism by Lemma~\ref{SplitFG}(i)) neutral splitting algebra $\scA_X$ of $\langle X\rangle \subset \bT$, which we consider as an object in $\Alg\bT$.

Now we define a functor
$$A:\,\mathbb{O}\,\to\,\Alg\bT$$
with the properties:
\begin{enumerate}[(a)]
\item For every $X\in\mO$, the algebra $A(X)$ is a neutral splitting algebra for $\langle\preceq X\rangle$.
\item $A(X_0)=\scA_{X_0}$, for the unique minimal element $X_0\in\mO$ under $\preceq$.
\item If we have a cover $Y\prec X$ in $\preceq$, then  $A(X)=A(Y)\otimes_{\Gamma_{\langle X\rangle}A(Y)}\scA_{X}$ for some algebra morphism $\Gamma_{\langle X\rangle}A(Y)\to\scA_{X}$ as in Lemma \ref{LemSlick}(i). The morphism $A(X)\to A(Y)$ is the canonical one.

\item If $X$ has no predecessor in $\mO$, then we consider $B:=\varinjlim_{Y\prec X}A(Y)$, which by construction is a neutral splitting algebra for $\langle\prec X\rangle$. Furthermore, we set $A(X):=B\otimes_{\Gamma_{\langle X\rangle}B}\scA_{X}$ for some algebra morphism $\Gamma_{\langle X\rangle}B\to\scA_{X}$ as in Lemma \ref{LemSlick}(i). We have canonical morphisms $A(Y)\to A(X)$ (which factor through $B$), for all $Y\prec X$ by construction.
\end{enumerate}
That such a functor exists follows from the principle of transfinite recursion (and the axiom of choice).
Now it follows that $\bT$ admits a neutral splitting algebra $\varinjlim_{\mO}A$, so by Theorem~\ref{ThmFibreNeut}, $\bT$ is neutral tannakian.

Now consider a tensor category $\bT$ with two neutral splitting algebras $\scA^1$ and $\scA^2$. We have corresponding functors $\mO\to\Alg\bT$ given by $X\mapsto \Gamma_{\langle\preceq X\rangle}\scA^i$. Using transfinite recursion as above, together with lemma \ref{LemSlick}(ii), we can construct a natural transformation between the two functors, which yields an algebra morphism between their colimits $\scA^1$ and $\scA^2$. By Theorem~\ref{ThmFibreNeut} the above observations show that between any two tensor functors $\bT\to\vecc$ there exists a natural transformation of tensor functors (hence an isomorphism).
\end{proof}

\begin{rem}
Together with \cite[Theorem~3.2(b)]{DM}, Theorem~\ref{NewThmNeut} demonstrates that an affine group scheme $G$ over an algebraically closed field $\mk$ admits no non-trivial torsors over $\mk$ (principal homogeneous spaces). For the special case with $G$ an algebraic group, this is demonstrated in \cite[\S 18.4]{Wa}.
\end{rem}


\subsection*{Acknowledgement} The author thanks Pavel Etingof, Akira Masuoka, Daniel Sch\"appi, Ross Street, Catharina Stroppel and Geordie Williamson for interesting discussions and Victor Ostrik and both referees for very useful comments on the previous versions of the manuscript. The author thanks Daniel Sch\"appi for spotting an error in the previous version of Theorem~\ref{ThmSuper} and helping to fix it. The research was partially carried out during a visit to the Max Planck Institute for Mathematics and supported by ARC grants DE170100623 and DP180102563.


\begin{thebibliography}
	{EHO}
	
	\bibitem[Al]{Alperin} J.L.~Alperin: Local representation theory. Cambridge Studies in Advanced Mathematics, 11. Cambridge University Press, Cambridge, 1986.	
	\bibitem[AK]{AK}
	Y.~Andr\'e, B.~Kahn:
Nilpotence, radicaux et structures mono\"idales.
With an appendix by Peter O'Sullivan. 
Rend. Sem. Mat. Univ. Padova 108 (2002), 107--291.

\bibitem[AGV]{AGV} M.~Artin, A.~Grothendieck, J.L.~Verdier: Th\'eorie des topos et cohomologie \'etale des sch\'emas (SGA4), vol. 1. Lecture Notes in Mathematics 269, Springer Verlag, 1972.

\bibitem[BE]{BE} D.~Benson, P.~Etingof: Symmetric tensor categories in characteristic 2. Adv. Math. 351 (2019), 967--999.

\bibitem[Be]{Ber}J.G.~Berkovic: On p-subgroups of finite symmetric and alternating groups. Representation theory, group rings, and coding theory, 67--76, 
Contemp. Math., 93, Amer. Math. Soc., Providence, RI, 1989. 

\bibitem[CE]{CE}K.~Coulembier, P.~Etingof: Maximal tannakian and super-tannakian subcategories of symmetric tensor categories. Appendix to P.~Etingof, S.~Gelaki, arXiv:1901.00528.
	
	\bibitem[De1]{Del90} P.~Deligne: Cat\'egories tannakiennes. The Grothendieck Festschrift, Vol. II, 111--195, Progr. Math., 87, Birkh\"auser Boston, Boston, MA, 1990. 
	
	\bibitem[De2]{Del02} P.~Deligne: Cat\'egories tensorielles. Mosc. Math. J. 2 (2002), no. 2, 227--248.
	
	\bibitem[De3]{Del11} P.~Deligne: Letter to A. Vasiu on November 30, 2011. (available on www.jmilne.org/math/)
	
	\bibitem[DM]{DM} P.~Deligne, J.S.~Milne: Tannakian Categories. In
Hodge cycles, motives, and Shimura varieties. 
Lecture Notes in Mathematics, 900. Springer-Verlag, Berlin-New York, 1982, pp. 101-228.
	
	\bibitem[EHO]{EHO} P.~Etingof, N.~Harman, V.~Ostrik: p-adic dimensions in symmetric tensor categories in characteristic p. Quantum Topol. 9 (2018), no. 1, 119--140.
	
	\bibitem[EVO]{EVO} P.~Etingof, V.~Ostrik, S.~Venkatesh: Computations in symmetric fusion categories in characteristic p. Int. Math. Res. Not. IMRN 2017, no. 2, 468--489.
	
	\bibitem[GK]{GK}  S.~Gelfand, D.~Kazhdan: Examples of tensor categories. Invent. Math. 109 (1992), no. 3, 595--617.
	
	\bibitem[GM]{GM} G.~Georgiev, O.~Mathieu:
	Cat\'egorie de fusion pour les groupes de Chevalley. 
C. R. Acad. Sci. Paris S\'er. I Math. 315 (1992), no. 6, 659--662.

	
	\bibitem[Jn]{Jantzen}
J.C.~Jantzen:
Representations of algebraic groups. 
Second edition. Mathematical Surveys and Monographs, 107. American Mathematical Society, Providence, RI, 2003.



 
\bibitem[Ma]{Mas}
A.~Masuoka: The fundamental correspondences in super affine groups and super formal groups. J. Pure Appl. Algebra 202 (2005), no. 1-3, 284--312.


\bibitem[Os]{Ostrik}
V.~Ostrik:
On symmetric fusion categories in positive characteristic.  arXiv:1503.01492.


\bibitem[Sc1]{Sch2}D.~Sch\"appi: Constructing colimits by gluing vector bundles.
arXiv:1505.04596.

\bibitem[Sc2]{SchappiNew} D.~Sch\"appi: Flat replacements of homology theories. arXiv:2011.12106.

\bibitem[Wa]{Wa}
W.C.~Waterhouse:
Introduction to affine group schemes. 
Graduate Texts in Mathematics, 66. Springer-Verlag, New York-Berlin, 1979.



\end{thebibliography}
\end{document}